%% file: main.tex
\documentclass{amsart}

\input{nir}
\usepackage[backend=biber,
    style=alphabetic,
    sorting=ynt
]{biblatex}
\title{Intertwining Operators for Siegel Parabolics over Finite Fields} 
\author{Nir Elber and Hahn Lheem}
\date{\today}

\begin{document}

\begin{abstract}
    We consider degenerate principal series representations $\Ind_P^G\chi$ over finite fields, where $G$ is a classical subgroup of $\operatorname{GL}_{2n}$, and $P$ is the Siegel parabolic subgroup. For example, we show that this representation is always multiplicity-free and irreducible for generic characters $\chi$. We then discuss a particular intertwining operator $I$ on $\Ind_P^G\chi$ and its related combinatorics. Firstly, this operator $I$ produces families of diagonalizable antitriangular matrices with well-behaved eigenvalues. Secondly, applying $I$ to a special vector in $\Ind_P^G\chi$ leads us to various matrix Gauss sums, whose evaluations imply an explicit equidistribution result of the trace and determinant of symmetric and alternating invertible matrices.
\end{abstract}

\maketitle


\tableofcontents

\section{Introduction}
For motivation, we begin by reviewing the doubling method for finite fields as worked out in \cite{chang-doubling}, though we remark that the notation of the introduction thus differs from the notation of the rest of the paper. We refer to \cite{ps-r-doubling} for the theory over local and global fields. Let $q$ be a prime power not divisible by $2$ or $3$, and let $2n$ be a positive even integer. The ``doubling method'' is a way to define zeta functions and gamma factors to arbitrary irreducible representations $\pi$ of a classical group; notably, we require no genericity assumption. For the purposes of the introduction, we work with the group $G=\op{SL}_n(\FF_q)$.

The main idea of the doubling method is to embed our classical group of one of the same type and twice the size. Thus, we let $H\coloneqq\op{SL}_{2n}(\FF_q)$, and we let $P\subseteq H$ denote the Siegel parabolic subgroup of matrices of the form $\begin{bsmallmatrix}
    A & B \\ & D
\end{bsmallmatrix}$ where $A,B,D\in\FF_q^{n\times n}$ and $\det AD=1$. 
Note that $G\times G$ embeds diagonally into $P$ via $(A,B)\mapsto\op{diag}(A,B)$. 
Then for a character $\omega$ of $P$, one tries to prove a multiplicity one result of the form
\[\dim\op{Hom}_{G\times G}\left(\op{Ind}_P^H\omega\otimes\pi\otimes\pi^\lor,\CC\right)\stackrel?=1.\]
This multiplicity one result allows one to define a zeta function $Z(f,v,w)$ for $\pi$. 
One would like this zeta function to have a functional equation, so we need a map from $\Ind_P^H\omega$ to a dual version. For this, we define a special intertwining operator $M\colon\Ind_P^G\omega\to\Ind_P^G\omega'$ by
\[Mf(g)\coloneqq\sum_{B\in\FF_q^{n\times n}}f\left(\begin{bmatrix}
    & 1_n \\ -1_n
\end{bmatrix}\begin{bmatrix}
    1_n & B \\ & 1_n
\end{bmatrix}g\right),\]
where $\omega'\colon P\to\CC^\times$ is some other explicitly defined character; for example, one has $\omega''=\omega$. In general, one has that $M\circ M$ is an operator on $\Ind_P^G\omega$; when $\omega^2=1$, it turns out that $\omega'=\omega$ so that $M$ is an operator on $\Ind_P^G\omega$.

This intertwining operator $M$ now provides a functional equation for $Z$ of the form
\[Z(Mf,v,w)=\Gamma(\pi,\omega)Z(f,v,w),\]
where $\Gamma(\pi,\omega)$ is our gamma factor; see \cite[Theorem~3.14]{chang-doubling}. One would like to normalize $\Gamma(\pi,\omega)$, which is typically done by using the functional equation twice to note that $\left|\Gamma(\pi,\omega)\right|^2$ should be an eigenvalue of $M\circ M$.

Thus, we are interested in knowing the eigenvalues of $M$. Because $M$ is $H$-invariant, it will be helpful to decompose $\op{Ind}_P^H\omega$ into irreducible components. For example, the following theorem counts those components. It follows by combining \Cref{prop:ind-irred,prop:ind-mult-free}. The proof uses Gelfand pairs and Mackey theory.
\begin{theorem}
    The representation $\Ind_P^H\omega$ is multiplicity-free as an $H$-representation. 
    Furthermore, the number of irreducible components equals
    \[\begin{cases}
        1 & \text{if }\omega^2\ne1, \\
        2 & \text{if }\omega^2=1\text{ but }\omega\ne1, \\
        n+1 & \text{if }\omega=1.
    \end{cases}\]
\end{theorem}
For example, we see that most $\omega$ make $\Ind_P^H\omega$ irreducible, so $M\circ M$ must be a scalar. In these cases, one way to proceed is to find a special vector in $\Ind_P^H\omega$ for which one can see directly that it is an eigenvector and compute this eigenvalue; see \cite[Section~3.6]{chang-doubling} for more discussion. This motivates the following result, which is \Cref{prop:i-on-psi-eigen}.
\begin{theorem}
    Fix a nontrivial character $\psi\colon\FF_q\to\CC^\times$. For each character $\omega\colon P\to\CC^\times$, we define a vector $f_\omega\in\Ind_P^H\omega$ as supported on matrices of the form $p\begin{bsmallmatrix}
        & -1_n \\ 1_n
    \end{bsmallmatrix}\begin{bsmallmatrix}
        1_n & B \\ & 1_n
    \end{bsmallmatrix}$ for $p\in P$ with value
    \[f_\omega\left(p\begin{bmatrix}
        & -1_n \\ 1_n
    \end{bmatrix}\begin{bmatrix}
        1_n & B \\ & 1_n
    \end{bmatrix}\right)\coloneqq\omega(p)\psi(\tr B).\]
    Then
    \[Mf_\omega=\Bigg(\sum_{B\in\GL_n(\FF_q)}\omega(\det B)\psi(\tr B)\Bigg)f_{\omega'}.\]
\end{theorem}
\begin{remark}
    In fact, even when $M$ fails to be a scalar, we will be able to show that the eigenvalue given by the Gauss sum equals the eigenvalue of smallest absolute value of $M$, and we conjecture that this eigenvalue has the largest eigenspace in $\Ind_P^H\omega$. This is discussed further in \Cref{rem:generic-eigenvalue}. It would be interesting to explicitly compute (or at least compare) the dimensions of all the eigenspaces of $\op{Ind}_P^H\omega$.
\end{remark}

Thus, we are motivated to evaluate these Gauss sums. In the case of $H=\SL_{2n}(\FF_q)$, the corresponding Gauss sums given as above have been evaluated in \cite{kim-gauss-sum}. However, considerations of other groups $G$ lead to different sums. For example, $H=\Sp_{2n}(\FF_q)$ leads to a sum over invertible symmetric matrices considered in \cite{saito-sym-gauss-sum}, and $H=\O_{4n}(\FF_q)$ leads to a sum over invertible alternating matrices (which appears to be new).

We provide evaluations for all of these matrix Gauss sums. Our methods are based on an explicit row-reduction analogous to the Bruhat decomposition methods of \cite{kim-gauss-sum}, but the explicit nature of our exposition allows our proofs to be rather uniform over all the various sums. For example, even though the sum over invertible symmetric matrices has already been considered in \cite{saito-sym-gauss-sum}, our method appears more direct.

Evaluating these Gauss sums also has a combinatorial application: we are able to provide an explicit formula for the number of invertible symmetric, alternating, or general matrices with given trace and determinant. For general invertible matrices, this application is essentially implicit in \cite[Theorem~6.2]{kim-gauss-sum}, so we only state results for symmetric and alternating matrices, which appear to be new. The following results follow from \Cref{cor:count-sym,cor:count-alt}, and they provide an explicit equidistribution result for the trace and determinant.
\begin{theorem}
    Fix $d\in\FF_q^\times$ and $t\in\FF_q$. For odd integers $2m+1$, the number $N(d,t)$ of symmetric $A\in\GL_{2m+1}(\FF_q)$ with $(\det A,\tr A)=(d,t)$ is bounded by
    \[\left|N(d,t)-\frac N{q(q-1)}\right|\le q^{m(m+1)}(q-1)^{m+1},\]
    where $N$ is the total number of invertible symmetric $(2m+1)\times(2m+1)$ matrices.
\end{theorem}
\begin{remark}
    There is an analogous, albeit slightly more complicated, result for even integers $2m$.
\end{remark}
\begin{theorem}
    Fix a positive integer $m$, a square $d\in\FF_q^{\times2}$, some $t\in\FF_q$, and an invertible $2m\times2m$ alternating matrix $T$. The number $N(d,t,T)$ of alternating $A\in\GL_{2m}(\FF_q)$ with $(\det A,\tr AT)=(d,t)$ is bounded by
    \[\left|N(d,t,T)-\frac N{q(q-1)/2}\right|\le q^{m(m-1)}(q-1)^m,\]
    where $N$ is the total number of invertible alternating $2m\times2m$ matrices.
\end{theorem}

We now return to our discussion of the eigenvalues of $M$. We have left to deal with some cases where $\Ind_P^H\omega$ fails to be irreducible. In this cases, we are able to write down a matrix representing $M$ and compute its eigenvalues. For simplicity, suppose $\omega=1$. By choosing a ``basis'' of $\Ind_P^H1$, we show the following in \Cref{prop:trivial-matrix-coeffs} and \Cref{thm:eigens-gl}.
\begin{theorem} \label{thm:sln-eigens}
    One can give $\big(\Ind_P^H1\big)^P$ an ordered basis so that the operator $M$ has matrix given by
    \[\left[(-1)^{i+j-n}q^{n^2-i^2+\binom{i+j-n}2}\frac{(q;q)_i^2}{(q;q)_{n-j}^2(q;q)_{i+j-n}}\right]_{i+j\ge n},\]
    where $(a;q)_n\coloneqq\prod_{i=0}^{n-1}\left(1-aq^i\right)$ is the $q$-Pochhammer symbol; here, $i,j\in\{0,\ldots,n\}$ are indices, and $i+j\ge n$ indicates that the matrix has zeroes when $i+j< n$. This matrix is diagonalizable and has eigenvalues given by
    \[\left\{(-1)^{n-i}q^{\binom n2+\binom{i+1}2}:0\le i\le n\right\}.\]
\end{theorem}
\begin{remark}
    Considerations with other classical groups $G$ produce other families of diagonalizable antitriangular matrices.
\end{remark}
This produces a family of diagonalizable ``antitriangular'' matrices. We are not aware of any general method to handle such diagonalization problems, and it does not appear clear a priori that the eigenvalues listed above should be so well-behaved. Diagonalizing certain antitriangular (satisfying a ``global antidiagonal property'') matrices has combinatorial applications in \cite{britnell-antitriangular}, and some aspects of our methods can be considered $q$-analogues of their arguments, but the analogy is weak. Notably, the family of matrices considered in \Cref{thm:sln-eigens} does not satisfy the global antidiagonal property.

\subsection{Layout}
In \Cref{sec:rep-theory}, we examine the representation theory of $\Ind_P^H\omega$ and explain where the combinatorial applications arise. In \Cref{sec:gsum}, we evaluate our matrix Gauss sums and provide the combinatorial applications. Lastly, in \Cref{sec:qcombo}, we provide the diagonalization of our intertwining operator.

\subsection{Acknowledgements}
This research was conducted during the University of Michigan REU during the summer of 2023; it was funded by the NSF RTG Number Theory and Representation Theory grant. The authors are particularly indebted to their advisors Elad Zelingher and Jialiang Zou for endlessly helpful advice and guidance in many aspects of this paper, from suggestions on the Hecke algebra to a plethora of helpful references. This project could not exist without them. The authors are also grateful to Ofir Gorodetsky for aid in proving certain $q$-identities used in the article, most notably by explaining how to use the packages \texttt{qZeil} and \texttt{qMultiSum}. As such, the authors are also grateful to the Research Institute for Symbolic Computation for access to those two packages. Large language models were used in the proof-reading stage of the paper; notably, a technical error was located and resolved in the proof of \Cref{prop:ind-mult-free} during proof-reading.

The first author would also like to thank various friends in the undergraduate mathematics department at the University of California at Berkeley, in particular Jad Damaj, Sophie McCormick, and Julie Shields for diverting conversations regarding diagonalizing antitriangular matrices. Lastly, the first author is most thankful to Hui Sun for consistent companionship.

\input{intertwining/rep}

\input{intertwining/gsum}

\input{intertwining/qcombo}

\section*{References}
\printbibliography[heading=none]

\end{document}

%% file: nir.tex
\newif\ifnirfiles
\nirfilesfalse
\newif\ifnirfancy
\makeatletter
\@ifclassloaded{amsart}{\nirfancyfalse}{\nirfancytrue}
\makeatother

\usepackage[margin=1in, marginparwidth=2cm]{geometry}
\usepackage[table,dvipsnames]{xcolor}
\usepackage{amsmath,amssymb,amsthm}
\usepackage{amsfonts}
\usepackage{asymptote}
\usepackage{cancel}
\usepackage{enumitem}
\usepackage{etoolbox}
\usepackage{graphicx}
\usepackage{mathdots}
\usepackage{mathtools} 
\usepackage{pgffor}
\usepackage{subfiles}
\usepackage{tikz-cd}
\usepackage{todonotes}
\usepackage{xparse}
\usepackage{xr}

\newtoggle{nirfancy}
\ifnirfancy
	\toggletrue{nirfancy}
\else
	\togglefalse{nirfancy}
\fi
\ifnirfancy
	\usepackage{footnotebackref}
	\usepackage{tocbibind}
	
	\usepackage{cabin}
	\usepackage[default]{cantarell}
\fi

\usepackage{fancyhdr}

\fancypagestyle{contentpage}{%
	\lhead{\textit{\rightmark}}
	\cfoot{\thepage}
}

\definecolor{wikipediadarkblue}{rgb}{0.023, 0.270, 0.676}
\providecommand{\nirpdftitle}{Intertwining Operators}
\ifnirfancy
	\hypersetup{
		colorlinks,
		citecolor=black,
		filecolor=black,
		linkcolor=wikipediadarkblue,
		urlcolor=wikipediadarkblue,
		pdftitle={\nirpdftitle},
		pdfauthor={Nir Elber}
	}
\else
	\AtBeginDocument{
		\hypersetup{
			citecolor=black,
			filecolor=black,
			linkcolor=wikipediadarkblue,
			urlcolor=wikipediadarkblue,
			pdftitle={\nirpdftitle},
			pdfauthor={Nir Elber}
		}
	}
\fi
\usepackage{hyperref,cleveref}

\numberwithin{equation}{section}
\def\equationautorefname#1#2\null{%
	(#2\null)%
}

\newif\ifnirdebug
\nirdebugfalse

\ifnirdebug
	\usepackage{showkeys}
\fi

\newcommand{\ZZ}{\mathbb Z}

\newcommand{\CC}{\mathbb C}
\newcommand{\FF}{\mathbb F}

\newcommand{\floor}[1]{\left\lfloor{#1}\right\rfloor}

\newcommand{\op}[1]{\operatorname{#1}}

\newcommand{\mc}[1]{\mathcal{#1}}

\newcommand{\ov}[1]{\overline{#1}}

\newcommand{\tr}{\operatorname{tr}}

\newcommand{\onto}{\twoheadrightarrow}

\newcommand{\Ind}{\operatorname{Ind}}
\newcommand{\Res}{\operatorname{Res}}
\newcommand{\Sym}{\operatorname{Sym}}
\newcommand{\Alt}{\operatorname{Alt}}
\newcommand{\Hom}{\operatorname{Hom}}
\newcommand{\End}{\operatorname{End}}
\newcommand{\GL}{\mathrm{GL}}
\newcommand{\SL}{\mathrm{SL}}
\newcommand{\GO}{\mathrm{GO}}
\renewcommand{\O}{\mathrm{O}}
\newcommand{\GSp}{\mathrm{GSp}}
\newcommand{\Sp}{\mathrm{Sp}}

\DeclareFontFamily{U}{dmjhira}{}
\DeclareFontShape{U}{dmjhira}{m}{n}{ <-> dmjhira }{}

\AtBeginDocument{%
	\mathchardef\ordinarycolon=\mathcode`:
	\mathcode`:="8000
}
\makeatletter
\newcommand{\coloncheck}{\@ifnextchar={\coloneqq\@gobble}{\ordinarycolon}}
\makeatother
\begingroup\lccode`~=`: \lowercase{\endgroup\let~}\coloncheck

\usepackage{marginnote}
\makeatletter
\long\def\@mn@@@marginnote[#1]#2[#3]{%
	\begingroup
		\ifmmode\mn@strut\let\@tempa\mn@vadjust\else
			\if@inlabel\leavevmode\fi
			\ifhmode\mn@strut\let\@tempa\mn@vadjust\else\let\@tempa\mn@vlap\fi
		\fi
		\@tempa{%
			\vbox to\z@{%
				\vss
				\@mn@margintest
				\if@reversemargin\if@tempswa
						\@tempswafalse
					\else
						\@tempswatrue
				\fi\fi

					\llap{%
						\vbox to\z@{\kern\marginnotevadjust\kern #3
							\vbox to\z@{%
								\hsize\marginparwidth
								\linewidth\hsize
								\kern-\parskip
								\marginfont\raggedleftmarginnote\strut\hspace{\z@}%
								\ignorespaces#1\endgraf
								\vss
							}%
							\vss
						}%
						\if@mn@verbose
							\PackageInfo{marginnote}{xpos seems to be \@mn@currxpos}%
						\fi
						\begingroup
							\ifx\@mn@currxpos\relax\else\ifx\@mn@currpos\@empty\else
									\kern\@mn@currxpos
							\fi\fi
							\ifx\@mn@currpage\relax
								\let\@mn@currpage\@ne
							\fi
							\if@twoside\ifodd\@mn@currpage\relax
									\kern-\oddsidemargin
								\else
									\kern-\evensidemargin
								\fi
							\else
								\kern-\oddsidemargin
							\fi
							\kern-1in
						\endgroup
						\kern\marginparsep
					}%
			}%
		}%
	\endgroup
}
\makeatother
\newlist{listalph}{enumerate}{1}
\setlist[listalph,1]{label=(\alph*)}
\newlist{listroman}{enumerate}{1}
\setlist[listroman,1]{label=(\roman*)}

\usepackage{thmtools,thm-restate}

\usepackage[framemethod=TikZ]{mdframed}
\usepackage[framemethod=TikZ]{mdframed}
\usepackage{xpatch}
\makeatletter
\xpatchcmd{\endmdframed}
	{\aftergroup\endmdf@trivlist\color@endgroup}
	{\endmdf@trivlist\color@endgroup\@doendpe}
	{}{}
\makeatother

\iftoggle{nirfancy}
{
	\definecolor{nirlightblue}{HTML}{f7f7ff}
	\definecolor{nirdarkblue}{HTML}{1d1dbf}
	\declaretheoremstyle[
		mdframed={
			backgroundcolor=nirlightblue,
			linecolor=nirdarkblue,
			rightline=false,
			topline=false,
			bottomline=false,
			linewidth=2pt,
			innertopmargin=5pt,
			innerbottommargin=8pt,
			innerleftmargin=8pt,
			leftmargin=-2pt,
			skipbelow=2pt,
			nobreak
		},
		headfont=\normalfont\bfseries\color{nirdarkblue}
	]{nirbluebox}
}
{
	\declaretheoremstyle[bodyfont=\itshape]{nirbluebox}
}
\ifx\thechapter\undefined
	\declaretheorem[style=nirbluebox,name=Theorem,within=subsection]{thm}
\else
	\declaretheorem[style=nirbluebox,name=Theorem,within=chapter]{thm}
\fi
\declaretheorem[style=nirbluebox,name=Theorem,numbered=no]{thm*}
\declaretheorem[style=nirbluebox,name=Theorem,sibling=thm]{theorem}
\declaretheorem[style=nirbluebox,name=Theorem,numbered=no]{theorem*}
\declaretheorem[style=nirbluebox,name=Proposition,sibling=thm]{prop}
\declaretheorem[style=nirbluebox,name=Proposition,numbered=no]{prop*}
\declaretheorem[style=nirbluebox,name=Proposition,sibling=thm]{proposition}
\declaretheorem[style=nirbluebox,name=Proposition,numbered=no]{proposition*}

\declaretheorem[style=nirbluebox,name=Lemma,numbered=no]{lem*}
\declaretheorem[style=nirbluebox,name=Lemma,sibling=thm]{lemma}
\declaretheorem[style=nirbluebox,name=Lemma,numbered=no]{lemma*}

\declaretheorem[style=nirbluebox,name=Corollary,numbered=no]{cor*}
\declaretheorem[style=nirbluebox,name=Corollary,sibling=thm]{corollary}
\declaretheorem[style=nirbluebox,name=Corollary,numbered=no]{corollary*}

\iftoggle{nirfancy}
{
	\definecolor{nirlightred}{RGB}{250, 220, 220}
	\definecolor{nirdarkred}{HTML}{f40000}
	\declaretheoremstyle[
		mdframed={
			backgroundcolor=nirlightred,
			linecolor=nirdarkred,
			rightline=false,
			topline=false,
			bottomline=false,
			linewidth=2pt,
			innertopmargin=5pt,
			innerbottommargin=8pt,
			innerleftmargin=8pt,
			leftmargin=-2pt,
			skipbelow=2pt,
			nobreak
		},
		headfont=\normalfont\bfseries\color{nirdarkred}
	]{nirredbox}
}
{
	\declaretheoremstyle[bodyfont=\itshape]{nirredbox}
}

\declaretheorem[style=nirredbox,name=Conjecture,numbered=no]{conj*}

\declaretheorem[style=nirredbox,name=Question,numbered=no]{ques*}

\declaretheorem[style=nirredbox,name=Convention,numbered=no]{conv*}

\declaretheorem[style=nirredbox,name=Convention,numbered=no]{convention*}

\iftoggle{nirfancy}
{
	\definecolor{nirdarkgreen}{HTML}{20660a}
	\definecolor{nirlightgreen}{HTML}{ebffeb}
	\declaretheoremstyle[
		mdframed={
			backgroundcolor=nirlightgreen,
			linecolor=nirdarkgreen,
			rightline=false,
			topline=false,
			bottomline=false,
			linewidth=2pt,
			innertopmargin=5pt,
			innerbottommargin=8pt,
			innerleftmargin=8pt,
			leftmargin=-2pt,
			skipbelow=2pt,
			nobreak
		},
		headfont=\normalfont\bfseries\color{nirdarkgreen}
	]{nirgreenbox}
}
{
	\declaretheoremstyle[bodyfont=\itshape]{nirgreenbox}
}

\declaretheorem[style=nirgreenbox,name=Axiom,numbered=no]{ax*}

\declaretheorem[style=nirgreenbox,name=Axiom,numbered=no]{axiom*}

\declaretheorem[style=nirgreenbox,name=Inventory,numbered=no]{inv*}

\declaretheorem[style=nirgreenbox,name=Inventory,numbered=no]{inventory*}
\declaretheorem[style=nirgreenbox,name=Notation,sibling=thm]{notation}
\declaretheorem[style=nirgreenbox,name=Notation,numbered=no]{notation*}

\declaretheorem[style=nirgreenbox,name=Definition,numbered=no]{defihelper*}

\iftoggle{nirfancy}
{
	\definecolor{nirdarkcyan}{HTML}{25805e}
	\definecolor{nirlightcyan}{HTML}{edfffb}
	\declaretheoremstyle[
		mdframed={
			backgroundcolor=nirlightcyan,
			linecolor=nirdarkcyan,
			rightline=false,
			topline=false,
			bottomline=false,
			linewidth=2pt,
			innertopmargin=5pt,
			innerbottommargin=8pt,
			innerleftmargin=8pt,
			leftmargin=-2pt,
			skipbelow=2pt,
			nobreak
		},
		headfont=\normalfont\bfseries\color{nirdarkcyan}
	]{nircyanbox}
}
{
	\declaretheoremstyle[bodyfont=\itshape]{nircyanbox}
}

\declaretheorem[style=nircyanbox,name=Example,numbered=no]{ex*}

\declaretheorem[style=nircyanbox,name=Example,numbered=no]{example*}

\declaretheorem[style=nircyanbox,name=Non-Example,numbered=no]{nex*}

\declaretheorem[style=nircyanbox,name=Non-Definition,numbered=no]{ndefi*}

\declaretheorem[style=nircyanbox,name=Exercise,numbered=no]{exercise*}

\declaretheorem[style=nircyanbox,name=Exercise,numbered=no]{exe*}

\iftoggle{nirfancy}
{
	\definecolor{nirlightbrown}{RGB}{252,240,235}
	\definecolor{nirdarkbrown}{HTML}{7a5342}
	\declaretheoremstyle[
		mdframed={
			backgroundcolor=nirlightbrown,
			linecolor=nirdarkbrown,
			rightline=false,
			topline=false,
			bottomline=false,
			linewidth=2pt,
			innertopmargin=5pt,
			innerbottommargin=8pt,
			innerleftmargin=8pt,
			leftmargin=-2pt,
			skipbelow=2pt,
			nobreak
		},
		headfont=\normalfont\bfseries\color{nirdarkbrown}
	]{nirbrownbox}
}
{
	\declaretheoremstyle[bodyfont=\itshape]{nirbrownbox}
}
\declaretheorem[style=nirbrownbox,name=Remark,sibling=thm]{remark}
\declaretheorem[style=nirbrownbox,name=Remark,numbered=no]{remark*}

\declaretheorem[style=nirbrownbox,name=Quote,numbered=no]{quot*}

\makeatletter
\def\thmt@setheadstyle#1{%
	\thmt@style@headstyle{%
		\def\NAME{\the\thm@headfont ##1}%
		\def\NUMBER{\bgroup\@upn{##2}\egroup}%
		\def\NOTE{\if=##3=\else\bgroup\thmt@space\the\thm@notefont(##3)\egroup\fi}%
		\def\NIRNOTE{\if=##3=\else##3\fi}%
	}%
	\def\thmt@tmp{#1}%
	\@onelevel@sanitize\thmt@tmp
	\ifcsname thmt@headstyle@\thmt@tmp\endcsname
		\thmt@style@headstyle\@xa{%
			\the\thmt@style@headstyle
			\csname thmt@headstyle@#1\endcsname
		}%
	\else
		\thmt@style@headstyle\@xa{%
			\the\thmt@style@headstyle
			#1%
		}%
	\fi
}
\renewenvironment{thmt@restatable}[3][]{%
	\thmt@toks{}
	\stepcounter{thmt@dummyctr}
	\long\def\thmrst@store##1{%
		\@xa\gdef\csname #3\endcsname{%
			\@ifstar{%
				\thmt@thisistheonefalse\csname thmt@stored@#3\endcsname
			}{%
				\thmt@thisistheonetrue\csname thmt@stored@#3\endcsname
			}%
		}%
		\@xa\long\@xa\gdef\csname thmt@stored@#3\@xa\endcsname\@xa{%
			\begingroup
			\ifthmt@thisistheone
				\thmt@rst@storecounters{#3}%
			\else
				\@xa\protected@edef\csname the#2\endcsname{%
					\thmt@trivialref{thmt@@#3}{??}}%
				\ifcsname r@thmt@@#3\endcsname\else
					\G@refundefinedtrue
				\fi
				\@xa\let\csname c@#2\endcsname=\c@thmt@dummyctr
				\@xa\let\csname theH#2\endcsname=\theHthmt@dummyctr
				\let\label=\thmt@gobble@label
				\let\index=\@gobble
				\let\ltx@label=\@gobble
				\def\thmt@restorecounters{}%
				\@for\thmt@ctr:=\thmt@innercounters\do{%
					\protected@edef\thmt@restorecounters{%
						\thmt@restorecounters
						\protect\setcounter{\thmt@ctr}{\arabic{\thmt@ctr}}%
					}%
				}%
				\thmt@trivialref{thmt@@#3@data}{}%
			\fi
			\ifthmt@restatethis
				\thmt@restatethisfalse
			\else
				\csname #2\@xa\endcsname\ifx\@nx#1\@nx\else[{#1}]\fi
			\fi
			\ifthmt@thisistheone
				\label{thmt@@#3}%
			\fi
			##1%
			\csname end#2\endcsname
			\ifthmt@thisistheone\else\thmt@restorecounters\fi
			\endgroup
		}
		\csname #3\@xa\endcsname\ifthmt@thisistheone\else*\fi
		\@xa\end\@xa{\@currenvir}
	}
	\thmt@collect@body\thmrst@store
}{%
}
\makeatother
\iftoggle{nirfancy}
{
	\declaretheoremstyle[
		mdframed={
			backgroundcolor=nirlightgreen,
			linecolor=nirdarkgreen,
			rightline=false,
			topline=false,
			bottomline=false,
			linewidth=2pt,
			innertopmargin=5pt,
			innerbottommargin=8pt,
			innerleftmargin=8pt,
			leftmargin=-2pt,
			skipbelow=2pt,
			nobreak
		},
		headfont=\normalfont\bfseries\color{nirdarkgreen},
		headformat={\NAME\ \NUMBER\NOTE{\nirindex{\NIRNOTE}}}
	]{nirgreenboxindexed}
}
{
	\declaretheoremstyle[bodyfont=\itshape]{nirgreenboxindexed}
}

\declaretheorem[style=nirgreenboxindexed,name=Definition,numbered=no]{defi*}
\declaretheorem[style=nirgreenboxindexed,name=Definition,sibling=thm]{definition}
\declaretheorem[style=nirgreenboxindexed,name=Definition,numbered=no]{definition*}

\newcommand{\zjl}[2][]{\if\relax\detokenize{#1}\relax{\color{blue}\vspace{0em}{ZJL:}#2}\else\ifx#1h\relax\else{\color{blue}\vspace{0em}{ZJL}#2}\fi\fi
}

%% file: intertwining/rep.tex

\section{Group-Theoretic Set-Up} \label{sec:rep-theory}
In this section, we set up the necessary representation theory to proceed with the results in the rest of the paper.

\subsection{Groups and Subgroups}
Let $q$ be an odd prime power, and let $2n$ be a positive even integer; for convenience, we will take $3\nmid q$, but this is used infrequently. Throughout, $G$ will be one of the groups $\{{\GL}_{2n},{\SL}_{2n},\GO_{2n},\O_{2n},\GSp_{2n},\Sp_{2n}\}$ over the finite field $\FF_q$. To explicate our orthogonal and symplectic groups, we fix
\[\varepsilon\coloneqq\begin{cases}
    +1 & \text{if }G\in\{\GO_{2n},\O_{2n}\}, \\
    -1 & \text{if }G\in\{\GSp_{2n},\Sp_{2n}\},
\end{cases}\qquad\text{and}\qquad J\coloneqq\begin{bmatrix}
    & \varepsilon1_n \\
    1_n
\end{bmatrix}\]
so that $G$ is defined to preserve the quadratic form given by $J$, possibly up to multiplier. In the cases where $G\in\{\GL_{2n},\SL_{2n}\}$, it will be convenient to define $\varepsilon\coloneqq-1$ as well. Here, the blank entries in $J$ indicate zeroes, a convention that will stay in place for the rest of the article. Throughout, when there are multiple groups $G$ involved, we will use a superscript $(\cdot)^G$; for example, $\varepsilon^{\GL_{2n}}=-1$.

Note that $G$ has split maximal torus $T$ given by the diagonal matrices. The degenerate principal series representations are induced from the Siegel parabolic subgroup
\[P\coloneqq\left\{\begin{bmatrix}
    A & B \\
      & D
\end{bmatrix}\in G\right\},\]
where $A,B,D$ are implicitly in $\FF_q^{n\times n}$, a convention that will remain in place for any expression in block matrix form as above. We let $U\subseteq P$ be the unipotent radical of $P$, and we let $M\subseteq P$ be the Levi subgroup so that $P=M\ltimes U$. Explicitly,
\[U=\left\{\begin{bmatrix}
    1_n & B \\
      & 1_n
\end{bmatrix}\in G\right\}\qquad\text{and}\qquad M=\left\{\begin{bmatrix}
    A &   \\
      & D
\end{bmatrix}\in G\right\}.\]
The various cases of $G$ provide more constraints on these two subgroups. For example, if $G\in\{\GO_{2n},\O_{2n}\}$, then $B$ above must be alternating; if $G\in\{\GSp_{2n},\Sp_{2n}\}$, then $B$ above must be symmetric. Similarly, if $G=\SL_{2n}$, then $\det D=(\det A)^{-1}$; if $G\in\{\O_{2n},\Sp_{2n}\}$, then $D=A^{-\intercal}$; and if $G\in\{\GO_{2n},\GSp_{2n}\}$, then $D=\lambda A^{-\intercal}$ for some $\lambda\in\FF_q^\times$. A quick computation with the definition of $G$ in the various cases reveals that these are the only constraints.

It will be helpful in the sequel to understand characters of $P$. In all cases, we are able to define a ``Siegel determinant'' $\chi_{\det}\colon P\to\FF_q^\times$ given by
\[\chi_{\det}\left(\begin{bmatrix}
    A & B \\
      & D
\end{bmatrix}\right)\coloneqq(\det D)^{-1}.\]
In the cases $G\in\{\GL_{2n},\GO_{2n},\GSp_{2n}\}$, there is an additional ``multiplier'' $m\colon P\to\FF_q^\times$ given by
\[\begin{cases}
    m\left(\begin{bmatrix}
        A & B \\
          & D
    \end{bmatrix}\right)=\det AD & \text{if }G=\GL_{2n}, \\
    m\left(\begin{bmatrix}
        \lambda A & B \\
          & A^{-\intercal}
    \end{bmatrix}\right)\coloneqq\lambda & \text{else}.
\end{cases}\]
For the remaining cases of $G$, we will define $m$ to just be the trivial character. Both $\chi_{\det}$ and $m$ are characters by a direct computation. It turns out that these are essentially the only characters.
\begin{lemma} \label{lem:decompose-character}
    Let $\chi\colon P\to\CC^\times$ be a character. Then $\chi=(\alpha\circ m)(\beta\circ\chi_{\det})$ for some characters $\alpha,\beta\colon\FF_q^\times\to\CC^\times$.
\end{lemma}
\begin{proof}
    This follows from an explicit computation of $[P,P]$ in all cases. One can use $3\nmid q$ to show that $U\subseteq[P,P]$, and it helps to know that $[\mathrm{SL}_n,\mathrm{SL}_n]=\mathrm{SL}_n$ for $q>3$.
\end{proof}
With a discussion of characters out of the way, we pick up the following notation, which we will use without comment in the sequel.
\begin{notation}
    Fix a group $P$ and a character $\chi\colon P\to\CC^\times$. For any representation $V$ of $P$, we let $V^\chi$ denote the subspace of $\chi$-eigenvectors. Explicitly,
    \[V^\chi\coloneqq\{v\in V:pv=\chi(p)v\text{ for all }p\in P\}.\]
\end{notation}

\subsection{Some Weyl Group Computations}
An argument similar to \cite[Example~17.88]{milne-alg-group} verifies that the diagonal subgroup $T$ of $G$ is always a maximal torus; namely, one can check that $C_G(T)=T$. Then an argument similar to \cite[Example~17.42]{milne-alg-group} verifies that $N_G(T)$ consists of permutation matrices (up to torus elements); alternatively, one can study the Weyl group of the relevant root system and then convert this back into permutation matrices by hand. In any case, we let $W$ denote the Weyl group of $G$, and we let $W_P$ denote the Weyl group of the Siegel parabolic subgroup $P$.

It will be useful to explicitly compute these Weyl groups. If $G\in\{\GL_{2n},\SL_{2n}\}$, then $W$ consists of the permutation matrices up to a sign. For each $w\in W$, we let $\sigma_w\in N_G(T)$ denote the corresponding permutation matrix, and we let $d_w\in T$ be a diagonal matrix with entries in $\pm1$ such that $\det d_w\sigma_w=1$. (The following arguments are independent of the choice of $d_w$.) The point is that $\{d_w\sigma_w\}_{w\in W}$ provides a set of representatives for $W$ in $G$.

We would like a similar description for $G\in\{\GO_{2n},\O_{2n},\GSp_{2n},\Sp_{2n}\}$. The following lemma, briefly, determines which permutation matrices actually belong to $G$, up to a diagonal element.
\begin{lemma} \label{lem:weyl-normal-form}
    Suppose $G\in\{\GO_{2n},\O_{2n},\GSp_{2n},\Sp_{2n}\}$. Let $\Sigma$ be the set of permutations $\sigma\in S_{2n}$ such that $\sigma(i+n)\equiv\sigma(i)+n\pmod{2n}$ for each $i$.
    \begin{enumerate}[label=(\alph*)]
        \item For each $w$ representing a class in $W$, there exists a unique permutation $\sigma\in\Sigma$ such that $w=d\sigma$ for some diagonal matrix $d$.
        \item For each $\sigma\in\Sigma$, there exists some diagonal matrix $d$ with entries in $\{\pm1\}$ such that $d\sigma\in G$.
    \end{enumerate}
\end{lemma}
\begin{proof}
    Checking (a) is a matter of determining which permutations live in $G$. Checking (b) comes down to writing down relations between the entries in $d$ enforced by $d\sigma\in G$. We refer to \cite[Exercise~7.16]{kirillov-lie-algebra} for more details.
\end{proof}
\begin{remark}
    For consistency, we provide a convenient choice of signs $d_w$ for $w\in W$. If $G\in\{\GO_{2n},\O_{2n}\}$, then $\varepsilon=1$, so $d_w\coloneqq1_{2n}$ will always work. If $G\in\{\GSp_{2n},\Sp_{2n}\}$, then one can put signs $d_w$ on the top-right quadrant of $\sigma_w$. Explicitly, we take $d_{\sigma(i)}=-1$ if $i\le n$ and $\sigma(i)>n$, and we take $d_{\sigma(i)}=1$ otherwise.
\end{remark}
We will use these explicit representatives of $W$ to provide explicit representatives of certain double quotients. For example, $W$ itself provides representatives of $B\backslash G/B$ by the Bruhat decomposition, where $B\subseteq G$ is a Borel subgroup containing $T$. We will be interested in $P\backslash G/P$.
\begin{lemma} \label{lem:compute-pgp}
    For each $r\in\{0,1,\ldots,n\}$, define
    \[\eta_r\coloneqq\begin{bmatrix}
        1_{n-r} \\ &&& \varepsilon1_r \\
        && 1_{n-r} \\
        & 1_r
    \end{bmatrix}.\]
    Then $\{\eta_0,\ldots,\eta_n\}\subseteq G$ provides a set of representatives of the double quotients $P\backslash G/P$.
\end{lemma}
\begin{proof}


    We define a function $\rho\colon G\to\{0,\ldots,n\}$ by $\rho\left(\begin{bsmallmatrix}
        A & B \\ C & D
    \end{bsmallmatrix}\right)\coloneqq\op{rank}C$. We will show that $\rho$ descends to a bijection $P\backslash G/P\to\{0,\ldots,n\}$, from which the result follows.
    
    Two of the required checks can be handled directly. Note that $\rho$ is surjective because $\rho(\eta_r)=r$ for each $r\in\{0,\ldots,n\}$. Additionally, an expansion of some $2\times2$ block matrices is able to show that $\rho$ actually descends to a function on $P\backslash G/P$.

    It remains to show that $\rho\colon P\backslash G/P\to\{0,
    \ldots,n\}$ is injective. Unwinding definitions, it is enough to show that $\rho(g)=r$ implies that $g\in P\eta_rP$. Choosing a Borel subgroup $B\subseteq P$ containing $T$, we may use the Bruhat decomposition to see that each coset in $B\backslash G/B$ is represented by an element of the Weyl group $W$. Thus, we may assume that $g=w=d_w\sigma_w$ where $d_w\in T$ and $\sigma_w$ is a permutation matrix. One now uses the permutations available in $P$ to show that $\sigma_w$ can be moved into $P\eta_rP$.
\end{proof}
\begin{remark}
    The above proof shows that the double cosets
    \[P\eta_rP=\left\{\begin{bmatrix}
        A & B \\ C & D
    \end{bmatrix}\in G:\op{rank}C=r\right\}\]
    are all (Zariski) locally closed. In fact, $P\eta_0P$ is (Zariski) closed, and $P\eta_nP$ is the only (Zariski) open double coset (it is defined by $\det C\ne0$).
\end{remark}


\subsection{Parabolic Induction}
In the sequel, we will be interested in the representations $\Ind_P^G\chi$ where $\chi\colon P\to\CC^\times$ is a character. We spend this subsection collecting a few facts about these representations. In particular, we will show that these representations are multiplicity-free and irreducible for ``general'' $\chi$.

We begin with the generic irreducibility of $\Ind_P^G \chi$.
\begin{proposition} \label{prop:ind-irred}
    Fix a character $\chi\colon P\to\CC^\times$, which we write as $\chi=(\alpha\circ m)(\beta\circ\chi_{\det})$. Then the dimension of $\End_G\Ind_P^G\chi$ equals
    \[\begin{cases}
        n+1 & \text{if }\beta=1, \\
        2 & \text{if }\beta^2=1,\beta\ne1\text{ and }G=\SL_{2n}, \\
        n+1 & \text{if }\beta^2=1,\beta\ne1\text{ and }G\in\{\O_{2n},\Sp_{2n}\}, \\
        \floor{\frac12(n+2)} & \text{if }\beta^2=1,\beta\ne1\text{ and }G\in\{\GO_{2n},\GSp_{2n}\}, \\
        1 & \text{else}.
    \end{cases}\]
    In particular, $\Ind_P^G\chi$ is irreducible provided $\beta^2\ne1$.
\end{proposition}
\begin{proof}
    We use Mackey theory in the form of \cite[Theorem~32.1]{bump-lie-group}. Namely, we are interested in computing the dimension of the space $\mc H$ of functions $f\colon G\to\CC$ satisfying
    \[f(p_1gp_2)=\chi(p_1)\chi(p_2)f(g)\]
    for all $p_1,p_2\in P$ and $g\in G$. Thus, any $f\in \mc H$ is uniquely determined by its values on representatives of the double cosets $P\backslash G/P$. As such, we define $f_r\in \mc H$ to be supported on $P\eta_rP$ defined by $f_r(\eta_r)\in\{0,1\}$, where we take $f_r(\eta_r)=1$ provided that this gives a well-defined function in $\mc H$. \Cref{lem:compute-pgp} implies that $\{f_r:f_r\ne0\}$ is a basis of $\mc H$.

    We are left computing the number of $r$ such that $f_r\in \mc H$ is well-defined with $f_r(\eta_r)=1$. Fix some $r$ for us to check. After some rearranging, it is enough to check that any
    $p\in P$ such that $\eta_rp\eta_r^{-1}\in P$ satisfies $\chi(p)=\chi\left(\eta_rp\eta_r^{-1}\right)$. Writing
    \[p\coloneqq\begin{bmatrix}
        A_1 & A_2 & B_1 & B_2 \\
        A_3 & A_4 & B_3 & B_4 \\
            &     & D_1 & D_2 \\
            &     & D_3 & D_4
    \end{bmatrix}\]
    to have the same block matrix dimensions as $\eta_r$, one can compute that $\eta_rp\eta_r^{-1}\in P$ if and only if $A_3=B_4=D_2=0$.
    Thus, $\chi(p)=\chi\left(\eta_rp\eta_r^{-1}\right)$ is equivalent to always having
    \[\chi\left(\begin{bmatrix}
        A_1 & \varepsilon B_2 &  B_1 & A_2 \\
            &  D_4 & \varepsilon D_3 \\
            &      &  D_1 \\
            &      &  B_3 & A_4
    \end{bmatrix}\right)\stackrel?=\chi\left(\begin{bmatrix}
        A_1 & A_2 & B_1 & B_2 \\
            & A_4 & B_3 &     \\
            &     & D_1 &     \\
            &     & D_3 & D_4
    \end{bmatrix}\right).\]
    By expanding out the definition of $\chi$, we find that this is equivalent to
    \[\beta(\det A_4)\stackrel?=\beta(\det D_4),\]
    where we take the convention that the ``empty'' matrix has determinant $1$. 
    \begin{itemize}
        \item If $r=0$, then $A_4$ and $D_4$ are empty, so the condition holds. Thus, we will take $r>0$ in the rest of our casework.
        \item If $\beta=1$, then the condition holds. Thus, we will take $\beta\ne1$ in the rest of our casework.
        \item Take $G=\GL_{2n}$. Because $r>0$, $\det$ is always surjective, and here there are no conditions on how $\det A_4$ and $\det D_4$ should relate to each other, so the condition never holds.
        \item Take $G=\SL_{2n}$. Because $r>0$, $\det$ will always be surjective. If $r=n$, then the condition $\det p=1$ becomes $\det A_4=\det D_4^{-1}$, so we get a contribution in this case only when $\beta^2=1$. Otherwise, $r\notin\{0,n\}$, so $\det A_4$ and $\det D_4$ can be arbitrary elements of $\FF_q^\times$ (our condition $\det p=1$ only requires $\det A_1D_4D_1A_4=1$), so the condition never holds.
        \item Take $G\in\{\O_{2n},\Sp_{2n}\}$. Then $A_4=D_4^{-\intercal}$, so we are requiring $\beta(\det A_4)^2=1$. Because $\det$ is surjective when $r>0$, nonzero $r$ contribute in this case exactly when $\beta^2=1$.
        \item Take $G\in\{\GO_{2n},\GSp_{2n}\}$. Then $A_4=m(p)D_4^{-\intercal}$, so we are requiring
        \[\beta(\det A_4)^2=\beta(m(p))^r.\]
        With $r>0$, the values $\det A_4$ and $m(p)$ are arbitrary elements of $\FF_q^\times$, so we would like for $\beta(x)^2=\beta(y)^r$ for any $x,y\in\FF_q^\times$. Taking $y=1$ shows that we will only get contributions in this case when $\beta^2=1$, and taking $x=1$ shows that we will only get contributions when $\beta^r=1$ too. However, with $\beta\ne1$, we see that $\beta^r=1$ only happens when $r$ is even.
    \end{itemize}
    Tallying the above cases completes the proof.
\end{proof}
\begin{remark}
    In the sequel, we will make frequent use of the basis $f_\bullet$ of $\mc H$.
\end{remark}
Even though it is not currently relevant to our discussion, we will want a similar Mackey theory computation in the sequel. This requires a definition.
\begin{definition} \label{def:chi-j}
    Note that $J$ normalizes $M$.
    Thus, for any character $\chi\colon P\to\CC^\times$, we define the character $\chi^J$ as the following composite.
	\[\arraycolsep=1.4pt\begin{array}{cccccccc}
		P &\onto& M &\stackrel J\cong& M &\stackrel\chi\to& \CC^\times \\
		\begin{bsmallmatrix}
			A & B \\ & D
		\end{bsmallmatrix} &\mapsto& \begin{bsmallmatrix}
			A \\ & D
		\end{bsmallmatrix} &\mapsto& \begin{bsmallmatrix}
			D \\ & A 
		\end{bsmallmatrix} &\mapsto& \chi\left(\begin{bsmallmatrix}
			D \\ & A 
		\end{bsmallmatrix}\right)
	\end{array}\]
\end{definition}
\begin{remark}
    One can check that $\left(\chi^J\right)^J=\chi$. Further, if $\beta=1$, then one can compute that $\chi^J=\chi$; alternatively, if we only have $\beta^2=1$ but $G\in\{\SL_{2n},\O_{2n},\Sp_{2n}\}$ so that $m=1$, then we still have $\chi^J=\chi$.
	
\end{remark}
\begin{proposition} \label{prop:twisted-ind-basis}
	Fix a character $\chi\colon P\to\CC^\times$, which we write as $\chi=(\alpha\circ m)(\beta\circ\chi_{\det})$. Then we compute a basis for $\left(\Ind_P^G\chi\right)^{\chi^J}$. 
    In particular, we find
	\[\dim\left(\Ind_P^G\chi\right)^{\chi^J}=\dim\left(\Ind_P^G\chi\right)^{\chi}.\]
\end{proposition}
\begin{proof}
	We proceed as in \Cref{prop:ind-irred}. For brevity, set $\mc H_J\coloneqq\left(\Ind_P^G\chi\right)^{\chi^J}$. Again, $f\in\mc H_J$ is uniquely determined by its values on representatives of $P\backslash G/P$, so we set $f_r\in\mc H_J$ to be supported on $P\eta_rP$ defined by $f_r(\eta_r)\in\{0,1\}$ where we take $f_r(\eta_r)=1$ whenever possible; thus, $\{f_r:f_r\ne0\}$ is a basis of $\mc H_J$.

	Continuing as in \Cref{prop:ind-irred}, we are checking which $f_r\in\mc H_J$ are well-defined with $f_r(\eta_r)=1$.
    Rearranging, it is enough to check that
    if $p\in P$ has $\eta_rp\eta_r^{-1}\in P$, we need $\chi(p)=\chi^J\left(\eta_rp\eta_r^{-1}\right)$. Writing
	\[p\coloneqq\begin{bmatrix}
        A_1 & A_2 & B_1 & B_2 \\
        A_3 & A_4 & B_3 & B_4 \\
            &     & D_1 & D_2 \\
            &     & D_3 & D_4
    \end{bmatrix}\]
    to have the same dimensions as $\eta_r$, we can then compute that $\eta_rp\eta_r^{-1} \in P$
    if and only if $A_3=B_4=D_2=0$. Thus, $\chi(p)=\chi^J\left(\eta_rp\eta_r^{-1}\right)$ is equivalent to always having
	\[\chi\left(\begin{bmatrix}
        A_1 & A_2 & B_1 & B_2 \\
            & A_4 & B_3 &     \\
            &     & D_1 &     \\
            &     & D_3 & D_4
    \end{bmatrix}\right)\stackrel?=\chi^J\left(\begin{bmatrix}
        A_1 & \varepsilon B_2 &  B_1 & A_2 \\
            &  D_4 & \varepsilon D_3 \\
            &      &  D_1 \\
            &      &  B_3 & A_4
    \end{bmatrix}\right).\]
	The result now follows from a similar casework on $G$ and $r$. We will not write out the casework in its entirety because a similar computation is recorded in \Cref{prop:ind-irred}. However, we will provide the answers.
	\begin{itemize}
		\item Suppose $G\in\{\O_{2n},\Sp_{2n}\}$.
        If $r=n$, then we always get a contribution; otherwise, we get contributions only when $\beta^2=1$.
		\item Suppose $G=\SL_{2n}$. We get a contribution when $r=n$ and when $\beta=1$. Lastly, we also get a contribution when $r=0$ and $\beta^2=1$.
		\item Suppose $G=\GL_{2n}$. We get a contribution when $r=n$ or when $\beta=1$ only.
		\item Suppose $G\in\{\GO_{2n},\GSp_{2n}\}$. We get a contribution when $r=n$ and when $\beta=1$; otherwise, we get an additional contribution when $\beta^2=1$ and $r\equiv n\pmod2$.
	\end{itemize}
	Tallying the above cases and comparing with \Cref{prop:ind-irred} completes the proof.
\end{proof}

We now show that $\Ind_P^G\chi$ is multiplicity-free.
\begin{proposition} \label{prop:ind-mult-free}
    For any character $\chi\colon P\to\CC^\times$, the representation $\Ind_P^G\chi$ is multiplicity-free.
\end{proposition}
\begin{proof}
    Write $\chi=(\alpha\circ m)(\beta\circ\chi_{\det})$. If $\beta^2\ne1$, then \Cref{prop:ind-irred} tells us that $\Ind_P^G\chi$ is irreducible. It remains to handle the case where $\beta^2=1$. Consider the Hecke algebra $\mc H$ of functions $f\colon G\to\CC$ satisfying
    \[f(p_1gp_2)=\chi(p_1)\chi(p_2)f(g).\]
    for all $p_1,p_2\in P$ and $g\in G$, where product is given by convolution. By \cite[Theorem~45.1]{bump-lie-group}, it suffices for the Hecke algebra $\mc H$ to be commutative. We will split this into three cases.
    \begin{itemize}
        \item Take $G=\SL_{2n}$ where $\chi\ne1$. We compute the Hecke algebra by hand. Here, $\alpha=1$, so we still have $\chi^2=1$. Then the computation of \Cref{prop:ind-irred} tells us that $\mc H$ has $\CC$-basis given by the functions $f_0,f_n\colon G\to\CC$ where $f_r$ is supported on $P\eta_rP$ with $f_r(\eta_r)=1$. To check that $\mc H$ is commutative, it is enough to verify that $f_0*f_n=f_n*f_0$. We will do this by explicit computation. It is enough to check that
        \[(f_0*f_n)(\eta_r)\stackrel?=(f_n*f_0)(\eta_r)\]
        for $r\in\{0,n\}$. For $\eta_0=1_{2n}$, both convolutions vanish because $f_0$ and $f_n$ have disjoint supports. For $\eta_n$, a similar comparison of supports finds that both sides equal $1$.
        
        
        \item Take $G\in\{\SL_{2n},\O_{2n},\Sp_{2n}\}$, except the above case. Again, $\alpha=1$, so $\chi^2=1$. We apply an argument similar to the theory of Gelfand pairs, such as in \cite[Theorem~45.2]{bump-lie-group}. Define $\iota\colon G\to G$ by $\iota(g)\coloneqq\op{diag}(1_n,\varepsilon1_n)g^{-1}\op{diag}(1_n,\varepsilon1_n)$. Then $\iota(\iota(g))=g$, and $\iota(gh)=\iota(h)\iota(g)$ for $g\in G$, and $\chi(\iota(p))=\chi(p)^{-1}=\chi(p)$ for $p\in P$.\footnote{This last identity crucially requires that $\chi^2=1$, which is why this argument fails when $G\in\{{\op{GL}_{2n}},{\op{GO}_{2n}},{\op{GSp}_{2n}}\}$.}
        Thus, we may define an operator $(\cdot)^\iota\colon\mc H\to\mc H$ by
        \[f^\iota(g)\coloneqq f(\iota(g)).\]
        Now, $(\cdot)^\iota$ is $\CC$-linear, and it can be checked to be an anti-automorphism from the fact $\iota(gh)=\iota(h)\iota(g)$.
        
        However, we claim that $(\cdot)^\iota$ is in fact the identity map on $\mc H$, from which it follows that $\mc H$ is commutative. Fix some $f\in\mc H$; we wish to show that $f^\iota=f$. By \Cref{lem:compute-pgp}, we see that $f$ is uniquely determined by its values on the $\eta_r$ for $r\in\{0,\ldots,n\}$ where $f_r\ne0$, so it is enough to check that $f\left(\iota(\eta_r)\right)=f(\eta_r)$. This holds because $\iota(\eta_r)=\eta_r$ by construction of $\iota$ and $\eta_r$.

        \item Take $G\in\{\GL_{2n},\GO_{2n},\GSp_{2n}\}$. Let $S\coloneqq\ker m$ so that $S\in\{\SL_{2n},\O_{2n},\Sp_{2n}\}$.
        We will show this case by reducing the claim from $G$ to $S$. Let $\mc H^S$ denote the Hecke algebra corresponding to the group $S$ and character $\chi^S\coloneqq\chi|_S$, and we will set $\mc H^G\coloneqq\mc H$ and $\chi^G\coloneqq\chi$. We will show that the ring $\mc H^S$ surjects onto $\mc H^G$, which shows that $\mc H^G$ is commutative.
        
        For each $r\in\{0,\ldots,n\}$, let $f_r^G\in\mc H^G$ and $f_r^S\in\mc H^S$ denote the functions on the corresponding group supported on the double coset of $\eta_r$ with $f_r^\bullet(\eta_r)=1$ whenever possible. Then the set of nonzero $f_r^\bullet$ forms a basis of $\mc H^\bullet$ as discussed in the proof of \Cref{prop:ind-irred}. In fact, a careful reading of the computation in \Cref{prop:ind-irred} shows that $f_r^G\ne0$ implies that $f_r^S\ne0$ for each $r$, so we may construct a $\CC$-linear surjection $\pi\colon\mc H^S\to\mc H^G$ by $\pi\colon f_r^S\mapsto f_r^G$.

        To complete the proof, we will show that $\pi$ is multiplicative. Fix indices $r,s,t\in\{0,\ldots,n\}$ with $f_r^G,f_s^G,f_t^G\ne0$, so it is enough to check that
        \[\left(f_r^G*f_s^G\right)(\eta_t)\stackrel?=\left(f_r^S*f_s^S\right)(\eta_t).\]
        Expanding out the convolution, we are being asked to show that
        \[\sum_{h\in P^G\backslash G}f_r^G\left(\eta_th^{-1}\right)f_s^G(h)\stackrel?=\sum_{h\in P^S\backslash S}f_r^S\left(\eta_th^{-1}\right)f_s^S(h),\]
        where $P^G\subseteq G$ and $P^S\subseteq S$ are the Siegel parabolic subgroups. We will show that these two sums are equal term-wise. Note that the number of terms on each side agree because the inclusion $S\subseteq G$ descends to a bijection $P^S\backslash S\to P^G\backslash G$.


        We now show that our sums are equal term-wise. Namely, we want to show that
        \[f_r^G\left(\eta_th^{-1}\right)f_s^G(h)\stackrel?=f_r^S\left(\eta_th^{-1}\right)f_s^S(h)\]
        for any $h\in S$. A computation of the supports shows that one side vanishes if and only if the other side vanishes. Otherwise,
        %
        we may assume that $f_r^S\left(\eta_th^{-1}\right)f_s^S(h)\ne0$. Then we can write $h=p_1\eta_sp_2$ and $\eta_th^{-1}=p_1'\eta_sp_2'$ for $p_1,p_2,p_1',p_2'\in P_S$, and an expansion of the definitions of $f_\bullet^S$ and $f_\bullet^G$ quickly show that both sides are equal.
        \qedhere
    \end{itemize}
\end{proof}
In the sequel, we will be interested in $G$-invariant operators on $\Ind^G_P\chi$, so it will be worth our time to provide some vectors which will behave like a basis for this space. The main idea is as follows.
\begin{lemma} \label{lem:basis-of-ind}
	Fix a character $\chi\colon P\to\CC^\times$. For each irreducible subrepresentation $\pi$ of $\Ind^G_P\chi$, the $\chi$-eigenspace of $\pi$ is one-dimensional.
\end{lemma}
\begin{proof}
	We are being asked to show that $\dim\Hom_P\left(\chi,\Res^G_P\pi\right)=1$. This follows by combining \Cref{prop:ind-mult-free} with Frobenius reciprocity.
\end{proof}
Thus, we note that we can understand operators on $\Ind^G_P\chi$ by merely understanding where they send a vector from each irreducible subrepresentation. Each irreducible subrepresentation contributes a unique basis element to $\left(\Ind^G_P\chi\right)^\chi$, so we may just understand how the operator behaves on $\left(\Ind^G_P\chi\right)^\chi$. But $\left(\Ind^G_P\chi\right)^\chi$ is exactly the underlying vector space of the corresponding Hecke algebra $\mc H$, so the computation of \Cref{prop:ind-irred} provides a basis for this space. 

\subsection{The Intertwining Operator}
We are now ready to introduce the principal operator of this paper, which is an operator $I$ on the space $\op{Mor}(G,\CC)=\Ind^G_11$ defined by
\[(If)(g)\coloneqq\sum_{u\in U}f\left(J^{-1}ug\right).\]
Note that $I\colon\Ind^G_11\to\Ind^G_11$ is $G$-invariant. In more typical notation, $I$ is the intertwining operator $M_J$, where we view $J$ as representing a Weyl group element. As the space $\Ind^G_11$ is too large, we are instead interested in the spaces $\Ind_P^G\chi$ where $\chi\colon P\to\CC^\times$ is some character. One can check that $I$ restricts to a $G$-invariant map $\Ind_P^G\chi\to\Ind_P^G\chi^J$.

This article is interested in understanding the linear transformation $I\colon\Ind^G_P\chi\to\Ind^G_P\chi^J$ and in particular the eigenvalues of the operator $I\circ I$. (Note $I\circ I$ is automatically diagonalizable because $\Ind_P^G\chi$ is multiplicity-free by \Cref{prop:ind-mult-free}.)
For later use, we would like to expand $I$ out as a matrix using the bases of \Cref{lem:basis-of-ind}, which we see makes $I$ into a linear transformation
\[\left(\Ind^G_P\chi\right)^\chi\to\left(\Ind^G_P\chi^J\right)^\chi,\]
both of which have explicit bases by the computations of \Cref{prop:ind-irred,prop:twisted-ind-basis}. Because we are interested in $I\circ I$ as well, we also want to compute the linear transformation
\[\left(\Ind^G_P\chi^J\right)^\chi\to\left(\Ind^G_P\chi\right)^\chi,\]
where we again have explicit bases.

To start, we begin with the easier generic case.
\begin{proposition} \label{prop:generic-intertwining}
	Fix a character $\chi\colon P\to\CC^\times$, which we write as $\chi=(\alpha\circ m)(\beta\circ\chi_{\det})$. Suppose $\beta^2\ne1$. Then let $\{f_0\}$ and $\left\{f_n^J\right\}$ be the bases of $\left(\Ind_P^G\chi\right)^\chi$ and $\left(\Ind_P^G\chi^J\right)^\chi$ described in \Cref{prop:ind-irred,prop:twisted-ind-basis}, respectively. Then
    \[\begin{cases}
        If_0=f_n^J, \\
        If_n^J=\beta(\varepsilon)^n\left|U\right|f_0.
    \end{cases}\]
	In particular, $I\circ I$ is the scalar $\beta(\varepsilon)^n\left|U\right|$.
\end{proposition}
\begin{proof}
	Certainly $If_0\in\op{span}\left\{f_n^J\right\}$ and $If_n^J\in\op{span}\{f_0\}$. We now do our computations separately.
    \begin{itemize}
        \item For $If_0$, we know $If_0=If_0(\eta_n)f_n^J$, so we want to compute
    	\[If_0(\eta_n)=\sum_{u\in U}f_0\left(J^{-1}u\eta_n\right).\]
        But $U\cap JP\eta_n^{-1}=\{1_{2n}\}$, so the summand vanishes unless $u=1_{2n}$, hence $If_0(\eta_n)=1$ follows.
        \item For $If_n^J$, we know $If_n^J=If_n^J(\eta_0)f_0$, so we want to compute
    	\[If_n^J(\eta_0)=\sum_{u\in U}f_n^J\left(J^{-1}u\eta_0\right).\]
        One may use the $P$-invariance of $f$ to rearrange the sum into $\left|U\right|f_n^J\left(J^{-1}\right)$. Computing with $f_n^J$ completes the proof.
        \qedhere
    \end{itemize}
\end{proof}

We now turn towards the case $\beta^2=1$. We begin with a general lemma.
\begin{lemma} \label{lem:matrix-coeff}
	Fix a character $\chi\colon P\to\CC^\times$, which we write as $\chi=(\alpha\circ m)(\beta\circ\chi_{\det})$. Given $r,s\in\{0,\ldots,n\}$ such that $f_r\in\left(\Ind_P^G\chi\right)^\chi$ (of \Cref{prop:ind-irred}) is nonzero, we have 
	\[If_r(\eta_s)=\beta(\varepsilon)^{n-s}Q\sum_{\substack{D\in\FF_q^{s\times s}\\\begin{bsmallmatrix}
		1_n & \op{diag}(D,0_{n-s}) \\ & 1_n
	\end{bsmallmatrix}\in G\\\op{rank}D=r+s-n}}\beta(\det E)^{-1},\]
	where
	\[Q\coloneqq\begin{cases}
		q^{n^2-s^2} & \text{if }G\in\{\GL_{2n},\SL_{2n}\}, \\
		q^{\binom{n}2-\binom{s}2} & \text{if }G\in\{\GO_{2n},\O_{2n}\}, \\
		q^{\binom{n+1}2-\binom{s+1}2} & \text{if }G\in\{\GSp_{2n},\Sp_{2n}\},
	\end{cases}\]
	and $E\in\GL_{r+s-n}(\FF_q)$ is some matrix determined from $D$ (not necessarily uniquely\footnote{See \Cref{rem:det-e-well-defined} below for a short discussion about well-definedness of the sum and in particular $\det E$.}) as follows:
    \begin{itemize}
        \item we always have $\begin{bsmallmatrix}
    		E \\ & 0
    	\end{bsmallmatrix}=D_1DD_2$ for $D_1,D_2\in\SL_s(\FF_q)$;
        \item and if $G\in\{\GO_{2n},\O_{2n}\}$, we require $D_2=D_1^\intercal$ and $E$ to be invertible and alternating;
        \item and if $G\in\{\GSp_{2n},\Sp_{2n}\}$, we require $D_2=D_1^\intercal$ and $E$ to be diagonal.
    \end{itemize}
\end{lemma}
\begin{proof}
	We are asked to compute $If_r(\eta_s)=\sum_{u\in U}f_r\left(J^{-1}u\eta_s\right)$. For this, we want to compute $J^{-1}u\eta_s$ and in particular want to ask when it lives in $P\eta_rP$. As such, we write $u$ in a block matrix form
    \[u=\begin{bmatrix}
			1_{n-s} && A & B \\ & 1_s & C & D \\
			&& 1_{n-s} \\ &&& 1_s
		\end{bmatrix}\]
    and compute
	\begin{align*}
		J^{-1}u\eta_s
		&= \varepsilon\begin{bmatrix}
			&& \varepsilon1_{n-s} \\ & \varepsilon1_s \\
			1_{n-s} &   &   \\ & D &   & \varepsilon1_s
		\end{bmatrix}\begin{bmatrix}
			1_{n-s} & B & A \\ & 1_s \\
			&& 1_{n-s} \\ && \varepsilon C & 1_s
		\end{bmatrix}.
	\end{align*}
	Now, $\chi$ is trivial on the last rightmost matrix,
    so we are left with
	\[If_r(\eta_s)=Q\sum_{D\in\FF_q^{s\times s}}f_r\left(\begin{bmatrix}
		&& 1_{n-s} \\ & 1_s \\
		\varepsilon1_{n-s} &   &   \\ &  D &   & 1_s
	\end{bmatrix}\right)\]
    upon replacing $D$ with $\varepsilon D$.
	Now, $f_r$ is supported on $P\eta_rP$, so by \Cref{lem:compute-pgp}, we see that $D$ gives a nonzero contribution if and only if \(\op{rank}\begin{bsmallmatrix}
		\varepsilon1_{n-s} \\ & D
	\end{bsmallmatrix}=r\), which is equivalent to $\op{rank}D=r+s-n$, which we will assume from now on. Set $d\coloneqq\op{rank}D$ for brevity.
	
	We now place $D$ into a normal form, which need not be unique.
	\begin{itemize}
		\item If $G\in\{\GL_{2n},\SL_{2n}\}$, then we use row-reduction to find matrices $D_1,D_2\in\SL_{s}(\FF_q)$ such that $D_1DD_2$ takes the form $\begin{bsmallmatrix}
			E \\ & 0
		\end{bsmallmatrix}$. 
		\item If $G\in\{\GSp_{2n},\Sp_{2n}\}$, then $D$ is symmetric, so finding an orthogonal basis grants $D_1\in\SL_s(\FF_q)$ such that $D_2\coloneqq D_1^\intercal$ has $D_1DD_2=\begin{bsmallmatrix}
			E \\ & 0
		\end{bsmallmatrix}$ where $E\in\GL_d(\FF_q)$ is diagonal.
		\item If $G\in\{\GO_{2n},\O_{2n}\}$, then $D$ is alternating, so finding a symplectic basis
        grants $D_1\in\SL_s(\FF_q)$ such that $D_2\coloneqq D_1^\intercal$ has $D_1DD_2=\begin{bsmallmatrix}
			E \\ & 0
		\end{bsmallmatrix}$ where $E$ is invertible and alternating.
	\end{itemize}
    Using the above normalizations, we may rewrite our summand as
    \[f_r\left(\begin{bmatrix}
		&& 1_{n-s} \\ & 1_s \\
		\varepsilon1_{n-s} &   &   \\ & D_1 DD_2 &   & 1_s
	\end{bmatrix}\right),\]
	reducing ourselves from $D$ to $D_1DD_2=\begin{bsmallmatrix}
		E \\ & 0
	\end{bsmallmatrix}$.
    We now note that $\begin{bsmallmatrix}
		1 \\ E & 1
	\end{bsmallmatrix}=\begin{bsmallmatrix}
		-\varepsilon E^{-1} & 1 \\ & E
	\end{bsmallmatrix}\begin{bsmallmatrix}
		& \varepsilon \\ 1
	\end{bsmallmatrix}\begin{bsmallmatrix}
		1 & E^{-1} \\ & 1
	\end{bsmallmatrix}$,
    so the summand equals
    \[\beta(\det E)^{-1}f_r\left(\begin{bmatrix}
		&&& 1_{n-s} \\ &&&& \varepsilon1_d \\ && 1_{n-r} \\
		\varepsilon1_{n-s} \\ & 1_d \\ && 0_{n-r} &&& 1_{n-r}
	\end{bmatrix}\right).\]
	To compute the contribution of this element, it remains to transform the middle matrix into $\eta_r$. We may factor out $\op{diag}(\varepsilon1_{n-s},1_s,\varepsilon1_{n-s},1_s)$ to turn the summand into
    \[\beta(\varepsilon)^{n-s}\beta(\det E)^{-1}f_r\left(\begin{bmatrix}
		&& \varepsilon1_r \\ & 1_{n-r} \\
		1_r \\ &&& 1_{n-r}
	\end{bmatrix}\right).\]
    We can now multiply by suitable permutation matrices to the above $4\times4$ block matrix to show that this equals $\beta(\varepsilon)^{n-s}\beta(\det E)^{-1}f_r(\eta_r)$, and summing completes the proof.
\end{proof}
\begin{remark} \label{rem:det-e-well-defined}
	Consider $G\in\{\GL_{2n},\SL_{2n}\}$ with $\beta\ne1$. In this case, the sum for $If_r(\eta_s)$ is often not well-defined as the value of $\det E$ fails to be well-defined given $D$ for most values of $r$ and $s$. This corresponds to the fact that we tend to have $f_r=0$ for most $r$. A similar phenomenon can be seen for the other groups. Note however that when $G\in\{\GO_{2n},\O_{2n}\}$ and $\beta^2=1$, we always have $\beta(\det E) = 1$ since the determinant of an alternating matrix is always a square.
\end{remark}
\begin{remark} \label{rem:i-on-fj}
	We can similarly describe $If_r^J(\eta_s)$, where $f_r^J\in\left(\Ind_P^G\chi^J\right)^{\chi}$ is the usual basis vector (of \Cref{prop:twisted-ind-basis}). If $\beta^2\ne1$, then we appeal to \Cref{prop:generic-intertwining}. Otherwise, if $\beta^2=1$, then the matrix factorizations used above apply verbatim, yielding the same sum.
\end{remark}
We are now in a position to write down some matrices when $\beta^2=1$. We first consider when $\beta=1$.
\begin{proposition} \label{prop:trivial-matrix-coeffs}
	Fix a character $\chi\colon P\to\CC^\times$, which we write as $\chi=(\alpha\circ m)(\beta\circ\chi_{\det})$. Suppose $\beta=1$ so that $\chi=\chi^J$. Then let $\{f_0,\ldots,f_n\}$ be the basis of $\left(\Ind_P^G\chi\right)^\chi$ described in \Cref{prop:ind-irred}. For each $i,j\in\{0,\ldots,n\}$, define
    \[\varepsilon(i,j)\coloneqq\begin{cases}
        \displaystyle(-1)^{i+j-n} & \text{if }G\in\{\GL_{2n},\SL_{2n}\}, \\
		\displaystyle(-1)^{(i+j-n)/2} & \text{if }G\in\{\GO_{2n},\O_{2n}\}, \\
		\displaystyle(-1)^{i+j-n-\floor{(i+j-n)/2}} & \text{if }G\in\{\GSp_{2n},\Sp_{2n}\}, \\
    \end{cases}\]
    and
    \[Q(i,j)\coloneqq\begin{cases}
        q^{n^2-i^2+\binom{i+j-n}2} & \text{if }G\in\{\GL_{2n},\SL_{2n}\}, \\
		q^{\binom{n}2-\binom{i}2+2\binom{(i+j-n)/2}2} & \text{if }G\in\{\GO_{2n},\O_{2n}\}, \\
		q^{\binom{n+1}2-\binom{i+1}2+2\binom{\floor{(i+j-n)/2}+1}2} & \text{if }G\in\{\GSp_{2n},\Sp_{2n}\}, \\
    \end{cases}\]
    and
    \[R(i,j)\coloneqq\begin{cases}
        \displaystyle\frac{(q;q)_i^2}{(q;q)_{n-j}^2(q;q)_{i+j-n}} & \text{if }G\in\{\GL_{2n},\SL_{2n}\}, \\
		\displaystyle\frac{(q;q)_i}{(q;q)_{n-j}(q^2;q^2)_{(i+j-n)/2}} & \text{if }G\in\{\GO_{2n},\O_{2n}\}, \\
		\displaystyle\frac{(q;q)_i}{(q;q)_{n-j}(q^2;q^2)_{\floor{(i+j-n)/2}}}& \text{if }G\in\{\GSp_{2n},\Sp_{2n}\},
    \end{cases}\]
	where we implicitly take zeroes unless $i+j-n$ is nonnegative and unless $i+j-n$ is even when $G\in\{\GO_{2n},\O_{2n}\}$. Then $[\varepsilon
    (i,j)Q(i,j)R(i,j)]_{0\le i,j\le n}$ is the matrix representation of $I$.
\end{proposition}
\begin{proof}
	We use \Cref{lem:matrix-coeff}, which applies because the $(i,j)$ matrix coefficient is given by $If_j(\eta_i)$. Because $\beta=1$, the coefficient equals the number of possible $D\in\FF_q^{i\times i}$ of rank $i+j-n$, with potentially additional conditions depending on $G$. If $G\in\{\GL_{2n},\SL_{2n}\}$, then we are counting all such matrices, so we appeal to \cite[Theorem~7.1.5]{hach-gf}. If $G\in\{\GO_{2n},\O_{2n}\}$, then we are counting alternating matrices, so we appeal to \cite[Theorem~7.5.5]{hach-gf}. Lastly, if $G\in\{\GSp_{2n},\Sp_{2n}\}$, then we are counting symmetric matrices, so we appeal to \cite[Theorem~7.5.2]{hach-gf}.
\end{proof}
We now turn to the case where $\beta^2=1$ but $\beta\ne1$.
For convenience, we explain how to reduce to the case $G\in\{\SL_{2n},\O_{2n},\Sp_{2n}\}$ by taking suitable submatrices.
\begin{lemma} \label{lem:general-from-special-matrix}
	Take $G\in\{\GL_{2n},\GO_{2n},\GSp_{2n}\}$ so that $S\coloneqq\ker m$ is in $\{\SL_{2n},\O_{2n},\Sp_{2n}\}$. Fix a character $\chi\colon P\to\CC^\times$, which we write as $\chi=(\alpha\circ m)(\beta\circ\chi_{\det})$.
    Suppose $\beta^2=1$ but $\beta\ne1$.
    Let $\left\{f_r^G\right\}_{r\in A}$ and $\left\{f_r^{JG}\right\}_{r\in B}$ be the bases of $\left(\Ind_P^G\chi\right)^\chi$ and $\left(\Ind_P^G\chi^J\right)^{\chi}$ described in \Cref{prop:ind-irred,prop:twisted-ind-basis} respectively; define $f_r^S$ and $f_r^{JS}$ similarly for $\chi|_S$. Further, let $\left[I^S(i,j)\right]_{0\le i,j\le n}$ be the matrix representation of $I$ on $\left(\Ind_{P^S}^S\chi\right)^\chi$.
	\begin{itemize}
		\item The matrix representation of $I^G\colon\left(\Ind_P^G\chi\right)^\chi\to\left(\Ind_P^G\chi^J\right)^\chi$ is
		\[\left[I^S(i,j)\right]_{\substack{0\le i,j\le n\\i\in B,j\in A}}.\]
		\item The matrix representation of $I^G\colon\left(\Ind_P^G\chi^J\right)^\chi\to\left(\Ind_P^G\chi\right)^\chi$ is
		\[\left[I^S(i,j)\right]_{\substack{0\le i,j\le n\\i\in A,j\in B}}.\]
	\end{itemize}
\end{lemma}
\begin{proof}
    For the first point, we see that as in \Cref{prop:trivial-matrix-coeffs}, the $(i,j)\in B\times A$ coefficient of $I^\bullet$ equals $I^\bullet f_j^\bullet(\eta_i)$. But the computation of \Cref{lem:matrix-coeff} explains that $I^Gf_j^G(\eta_i)=I^Sf_j^S(\eta_i)$, as required.
	The proof of the second point is essentially the same upon replacing \Cref{lem:matrix-coeff} with \Cref{rem:i-on-fj}.
\end{proof}
\begin{remark}
    It is worth recalling $A$ and $B$. If $G=\GL_{2n}$, then $A=\{0\}$ and $B=\{n\}$. Otherwise if $G\in\{\GO_{2n},\GSp_{2n}\}$, then $A=\{r:r\equiv0\pmod2\}$ and $B=\{r:r\equiv n\pmod2\}$.
\end{remark}
\begin{remark} \label{rem:special-coefs-vanish}
    Using the notation of the lemma above, we can show that $I^S(i,j)=0$ if $i\notin B$ but $j\in A$. Indeed, the argument above implies
	\[I^S(i,j)=I^Gf_j^G(\eta_i),\]
	which vanishes because $f_i^{JG}=0$. Similarly, we find $I^S(i,j)=0$ if $i\notin A$ but $j\in B$ by using $If_j^{JG}(\eta_i)=0$ instead.
\end{remark}
We are now ready for our computation.
\begin{proposition} \label{prop:quadratic-matrix}
    Take $G\in\{\SL_{2n},\O_{2n},\Sp_{2n}\}$. Fix a character $\chi\colon P\to\CC^\times$, which we write as $\chi=\beta\circ\chi_{\det}$.
    Suppose $\beta^2=1$ but $\beta\ne1$ so that $\chi=\chi^J$. Let $\{f_r\}_{r\in A}$
    be the basis of $\left(\Ind_P^G\chi\right)^\chi$ described in \Cref{prop:ind-irred}.
    \begin{itemize}
        \item If $G=\SL_{2n}$ so that $A=\{0,n\}$, then $I$ has the matrix representation
        \[\begin{bmatrix}
            & \beta(-1)q^{n^2} \\ 1
        \end{bmatrix}.\]
        \item If $G\in\{\O_{2n},\Sp_{2n}\}$ so that $A=\{0,\ldots,n\}$, then define
        \[\varepsilon(i,j)\coloneqq\beta(\varepsilon)^{n-i+(i+j-n)/2}(-1)^{(i+j-n)/2},
        \]
        and
        \[Q(i,j)\coloneqq\begin{cases}
            q^{\binom{n}2-\binom{i}2+2\binom{(i+j-n)/2}2} & \text{if }G=\O_{2n}, \\
            q^{\binom{n+1}2-\binom{i+1}2+2\binom{(i+j-n)/2+1}2-(i+j-n)/2} & \text{if }G=\Sp_{2n},
        \end{cases}\]
        and
        \[R(i,j)\coloneqq\frac{(q;q)_i}{(q;q)_{n-j}(q^2;q^2)_{(i+j-n)/2}},
        \]
        which vanish unless $i+j-n$ is a nonnegative even integer. Then $[\varepsilon(i,j)Q(i,j)R(i,j)]_{0\le i,j\le n}$ is the matrix representation of $I$.
    \end{itemize}
\end{proposition}
\begin{proof}
    \Cref{rem:special-coefs-vanish} explains all the vanishing entries. We now handle our groups separately.
    \begin{itemize}
        \item Take $G=\SL_{2n}$. Then it remains to compute $If_0(\eta_n)f_n^J$ and $If_n^J(\eta_0)f_0$.
    	\begin{itemize}
    		\item For $If_0(\eta_n)$, \Cref{lem:matrix-coeff} wants us to sum over $D\in\FF_q^{n\times n}$ of rank $0$. This requires $D=0$, so the sum yields $1$.
    		\item For $If_n^J(\eta_0)$, \Cref{lem:matrix-coeff} applies via  \Cref{rem:i-on-fj}. This time, we are summing $D\in\FF_q^{0\times0}$ of rank $0$, so the sum still returns $1$, yielding $If_n^J(\eta_0)=\beta(\varepsilon)^nq^{n^2}$.
    	\end{itemize}
        \item Take $G=\O_{2n}$. Using \Cref{lem:matrix-coeff}, we see $\beta(\det E)^{-1}=\beta(\varepsilon)=1$ always, so the same argument as in \Cref{prop:trivial-matrix-coeffs} goes through.
        \item Take $G=\Sp_{2n}$. Using \Cref{lem:matrix-coeff}, we see that our sum is the difference between the number of symmetric $D\in\FF_q^{i\times i}$ of rank $i+j-n$ with $\det E\in\FF_q^{\times2}$ and the number of such $D$ with $\det E\notin\FF_q^{\times2}$. The formulae of \cite{macwilliams-ortho-matrices} tell us that the number of such $D$ with $\det E\in\FF_q^{\times2}$ is
        \[\frac12N\cdot\frac{q^{(i+j-n)/2}+\beta(-1)^{(i+j-n)/2}}{q^{(i+j-n)/2}}\]
        when $i+j-n$ is even,
    	where $N$ is the total number of symmetric matrices $D\in\FF_q^{i\times i}$ of rank $i+j-n$. Thus, we see that the desired difference is $\beta(-1)^{(i+j-n)/2}q^{-(i+j-n)/2}N$. Plugging into \Cref{lem:matrix-coeff} completes the proof.
        \qedhere
    \end{itemize}
\end{proof}

\subsection{A Multiplicity One Result}
Our understanding of $I$ so far has relied on eigenvectors of $\Ind_P^G\chi$ with eigenvalue $\chi$ (or $\chi^J$). In this subsection, we will use eigenvectors
for the smaller subgroup $U\subseteq P$.
\begin{definition}
	Fix $T\in \FF_q^{n\times n}$ and a character $\psi\colon\FF_q\to\CC^\times$. Then we define the character $\psi_T\colon U\to\CC^\times$ by
	\[\psi_T\left(\begin{bmatrix}
		1_n & B \\ & 1_n
	\end{bmatrix}\right)\coloneqq\psi(\tr BT).\]
\end{definition}
\begin{definition}\label{def:produce-psi-t-eigen} 
	Fix $T\in \FF_q^{n\times n}$ and a character $\psi\colon\FF_q\to\CC^\times$. Given a character $\chi\colon P\to\CC^\times$, define $f_{\chi,T}\in\left(\Ind_P^G\chi\right)^{\psi_T}$ to be supported on $P\eta_nP$ and defined by
	\[f_{\chi,T}(p\eta_nu)\coloneqq\chi(p)\psi_T(u).\]
	One can show that any $g\in P\eta_nP$ can be written uniquely in the form $p\eta_nu$ where $p\in P$ and $u\in U$, so this is a well-defined function.
\end{definition} 
Here is the main result of the present subsection.
\begin{proposition} \label{prop:psi-t-mult-one}
	Fix $T\in\GL_n(\FF_q)$ such that $\begin{bsmallmatrix}
		1 & T \\ & 1
	\end{bsmallmatrix}\in G$ and a nontrivial character $\psi\colon\FF_q\to\CC^\times$. For any character $\chi\colon P\to\CC^\times$, we have
	\[\dim\op{Hom}_U\left(\psi_T,\Ind_P^G\chi\right)=1.\]
	In other words, $\op{Hom}_U\left(\psi_T,\Ind_P^G\chi\right)$ is spanned by the $f_{\chi,T}$, defined in \Cref{def:produce-psi-t-eigen}.
\end{proposition}
\begin{proof}
	We use Mackey theory. By Frobenius reciprocity, we are computing the dimension of the space $\Hom_G\left(\Ind_U^G\psi_T,\Ind_P^G\chi\right)$, which \cite[Theorem~32.1]{bump-lie-group} explains is isomorphic to the space $\mc H$ of functions $f\colon G\to\CC$ such that
    \[f(pgu)=\chi(p)f(g)\psi_T(u)\]
    for $p\in P$ and $u\in U$. We will prove that $\mc H$ is spanned by $f_{\chi, T}$ in four steps.
	\begin{enumerate}
		\item Observe any $f \in \mc H$ is determined by its values on representatives of the double coset space $P\backslash G/U$. \Cref{lem:compute-pgp} tells us that the double cosets $P\backslash G/P$ are represented by $\{\eta_0,\ldots,\eta_n\}$. As $P=MU$, the double coset space $P\backslash G/U$ is represented (not uniquely) by the set
        \[\{\eta_rd:0\le r\le n,d\in M\}.\]
		For the remainder of the proof, our goal will be to show that $f\in\mc H$ will have $f(\eta_rd)=0$ for any $d\in M$ whenever $r\ne n$. This will complete the proof because it shows that any $f\in\mc H$ is supported on $P\eta_nP=P\eta_nU$, meaning that $f=f(\eta_n)f_{\chi,T}$, so $\left\{f_{\chi,T}\right\}$ is a basis of $\mc H$. 

		The basic sketch is that we will find various $u\in U$ such that $\eta_rdu=p\eta_rd$ for some $p\in P$, which will allow us to show that $f(\eta_rd)=f(\eta_rdu)$, but then $f(\eta_rd)\ne0$ would imply $\psi_T(u)=1$. Having many such $u$ will allow us to force a full column of $T$ to vanish, violating the hypothesis that $T$ is invertible.


		\item Fix some $\eta_r$ and $d\in M$. If $\eta_rdu=p\eta_rd$ for some $u\in U$ and $p\in P$, then we claim $\chi(p)=1$. In other words, we are showing that $\chi$ is trivial on any $p\in P\cap\eta_rdUd^{-1}\eta_r^{-1}$. As $M$ normalizes $U$, we may reduce to the case $d=1_{2n}$.

		Suppose $u\in U$ satisfies $p\coloneqq\eta_ru\eta_r^{-1} \in P$; we want to show that $\chi(p)=1$. We can expand $u$ as
        \[u=\begin{bmatrix}
            1_{n-r} && A & B \\ & 1_r & C & D \\
				&& 1_{n-r} \\ &&& 1_r
        \end{bmatrix}\]
        and then compute
		\begin{align*}
			p &= \begin{bmatrix}
				1_{n-r} & \varepsilon B & A \\ & 1_r \\
				&& 1_{n-r} \\ & \varepsilon D & C & 1_r
			\end{bmatrix}.
		\end{align*}
		Then $p\in P$ is equivalent to $D=0$, in which case one can directly compute $\chi(p)=1$.

		\item Fix some $\eta_r$ and $d\in M$ such that $r<n$. We claim that there exists $u\in U$ such that $\psi_T(u)\ne1$ and $\eta_rdu=p\eta_rd$ for some $p\in P$. Assuming the contrary, we will show that having $\psi_T(u)=1$ for all such $u$ implies $T$ fails to be invertible. This is the only step of the proof which will use the invertibility of $T$ and nontriviality of $\psi$.

		The condition on $u\in U$ is that $\eta_rdud^{-1}\eta_r^{-1}\in P$. Because $M$ normalizes $U$, our hypothesis is simply that $\eta_ru\eta_r^{-1}\in P$ implies $\psi_T\left(d^{-1}ud\right)=1$. By replacing $T$ with $DTA^{-1}$, where $d = \op{diag}(A,D)$, we reduce to the case $d=1_{2n}$. In the computation of the previous step, we found many $u$ with $\eta_ru\eta_r^{-1}\in P$; for example, using $r<n$, we know $\psi_T$ must be trivial on
		\[U_1\coloneqq\left\{\begin{bmatrix}
			1 & B \\ & 1
		\end{bmatrix}\in G:B_{ij}=0\text{ for }i,j>1\right\}.\]
		We claim that the first column of $T$ is zero, which
        implies $T$ is not invertible. Expanding the trace of $BT$ shows that $\psi_T(u)=1$ for $u = \begin{bsmallmatrix}
            1 & B \\ & 1
        \end{bsmallmatrix}\in U_1$ is simply asserting
        \[1=\psi(T_{11}B_{11})\cdot \prod_{i=2}^n\psi(T_{i1}B_{1i})\cdot \prod_{j=2}^n\psi(T_{1j}B_{j1}).\]
		To continue, we must do some casework on $G$. If $G\in\{\GL_{2n},\SL_{2n}\}$, then we may set the $B_{ij}$ arbitrarily, provided $i=1$ or $j=1$. We would like to show that $T_{i1}=0$ for all $i$, so fixing some $i$, we set all coordinates except $B_{1i}$ to zero to force $\psi(T_{i1}B_{1i})=1$ for all $B_{1i}\in\FF_q$. This successfully implies $T_{i1}=0$ because $\psi$ is nontrivial. The arguments for other $G$ are essentially the same, except we must keep track of the requirement that $T$ and $B$ are alternating for $G\in\{\GO_{2n},\O_{2n}\}$ and symmetric for $G\in\{\GSp_{2n},\Sp_{2n}\}$.

		\item We now complete the proof. Given some $f\in\mc H$, we would like to show that $f(\eta_rd)=0$ whenever $r<n$. The previous step provides $p\in P$ and $u\in U$ such that $p\eta_rd=\eta_rdu$ and $\psi_T(u)\ne1$. But any such $p$ must have $\chi(p)=1$ by the second step, so the equation
		\[\chi(p)f(\eta_rd)=\psi_T(u)f(\eta_rd)\]
		forces $f(\eta_rd)=0$, as claimed.
		\qedhere
	\end{enumerate}
\end{proof}
\begin{remark} \label{rem:bad-orthogonal}
	It is possible for no $T$ satisfying the hypotheses of \Cref{prop:psi-t-mult-one} to exist! Namely, suppose $n$ is odd and $G\in\{\GO_{2n},\O_{2n}\}$. Then we are asking for $T$ to be an invertible $n\times n$ alternating matrix, which is impossible! However, we can find some $T$ in all other cases.
\end{remark}
This multiplicity-one result means that we can gain insight into $I\colon\Ind_P^G\chi\to\Ind_P^G\chi^J$ by plugging in $f_{\chi,T}$. This will lead us to evaluate certain matrix Gauss sums.
\begin{definition}
	Fix $T\in \FF_q^{n\times n}$ and characters $\beta\colon\FF_q^\times\to\CC^\times$ and $\psi\colon\FF_q\to\CC^\times$. Then we define the ``matrix Gauss sum''
	\[g^G(\beta,\psi,T)\coloneqq\sum_{\substack{B\in\GL_n(\FF_q)\\\begin{bsmallmatrix}
		1 & B \\ & 1
	\end{bsmallmatrix}\in G}}\beta(\det B)\psi(\tr BT).\]
\end{definition}
\begin{proposition} \label{prop:i-on-psi-eigen}
	Fix $T\in\GL_n(\FF_q)$ such that $\begin{bsmallmatrix}
		1 & T \\ & 1
	\end{bsmallmatrix}\in G$ and a nontrivial character $\psi\colon\FF_q\to\CC^\times$. Further, fix a character $\chi\colon P\to\CC^\times$, which we write as $\chi=(\alpha\circ m)(\beta\circ\chi_{\det})$. Then
	\[If_{\chi,T}=g^G(\beta,\psi,T)f_{\chi^J,T}.\]
\end{proposition}
\begin{proof}
	For brevity, let $\ov U$ denote the subgroup of $B\in \FF_q^{n\times n}$ such that $\begin{bsmallmatrix}
		1 & B \\ & 1
	\end{bsmallmatrix}\in G$, and we let $\ov U^\times$ denote the invertible subset. Note that $I$ carries $\psi_T$-eigenvectors to $\psi_T$-eigenvectors, so \Cref{prop:psi-t-mult-one} tells us that
	\[If_{\chi,T}=\left(If_{\chi,T}(\eta_n)\right)f_{\chi^J,T}.\]
	It remains to evaluate $If_{\chi,T}(\eta_n)$, which we do directly. To begin, note
	\[If_{\chi,T}(\eta_n)=\sum_{u\in U}f_{\chi,T}\left(\eta_n^{-1}u\eta_n\right).\]
	Now, writing $u=\begin{bsmallmatrix}
		1 & B \\ & 1
	\end{bsmallmatrix}$, we see that $\eta_n^{-1}u\eta_n=\begin{bsmallmatrix}
		1 \\ \varepsilon B & 1
	\end{bsmallmatrix}$, so
	\[If_{\chi,T}(\eta_n)=\sum_{B\in\ov U}f_{\chi,T}\left(\begin{bmatrix}
		1_n \\ B & 1_n
	\end{bmatrix}\right).\]
	Because $f_{\chi,T}$ is supported on $P\eta_nU=P\eta_nP$, the proof of \Cref{lem:compute-pgp} tells us that $B\in\ov U$ produces a nonzero contribution if and only if $B$ is invertible. To compute this contribution, we note $\begin{bsmallmatrix}
	    1 \\ B & 1
	\end{bsmallmatrix}=\begin{bsmallmatrix}
	    -\varepsilon B^{-1} & 1 \\ & B
	\end{bsmallmatrix}\begin{bsmallmatrix}
	    & \varepsilon \\ 1
	\end{bsmallmatrix}\begin{bsmallmatrix}
	    1 & B^{-1} \\ & 1
	\end{bsmallmatrix}$,
    so
	\[If_{\chi,T}(\eta_n)=\sum_{B\in\ov U^\times}\beta\left(\det B^{-1}\right)\psi_T\left(B^{-1}\right).\]
	Replacing $B$ with $B^{-1}$ completes the proof.
\end{proof}
Thus, we see that the values of $g^G(\beta,\psi,T)$ will be interesting to us. For example, when $\chi=\chi^J$, we see that $g^G(\beta,\psi,T)$ is an eigenvalue of $I$. In the general case when merely $I\circ I$ is an operator on $\Ind_G^P\chi$, we get the following.
\begin{corollary}
	Fix $T\in\GL_n(\FF_q)$ such that $\begin{bsmallmatrix}
		1 & T \\ & 1
	\end{bsmallmatrix}\in G$ and a nontrivial character $\psi\colon\FF_q\to\CC^\times$. Further, fix a character $\chi\colon P\to\CC^\times$, which we write as $\chi=(\alpha\circ m)(\beta\circ\chi_{\det})$. Then
	\[(I\circ I)f_{\chi,T}=\beta(-1)^n\left|g^G(\beta,\psi,T)\right|^2f_{\chi,T}.\]
\end{corollary}
\begin{proof}
	Applying \Cref{prop:i-on-psi-eigen} twice, we see that
	\[(I\circ I)f_{\chi,T}=g^G(\beta,\psi,T)g^G\left(\beta^{-1},\psi,T\right)f_{\chi,T}.\]
    This scalar equals $\beta(-1)^n\left|g^G(\beta,\psi,T)\right|^2$.
\end{proof}
\begin{remark}
	Suppose $\beta^2\ne1$, and we compare the above computation with \Cref{prop:generic-intertwining}. When $G\in\{\GL_{2n},\SL_{2n},\GSp_{2n},\Sp_{2n}\}$, we see that $\varepsilon=-1$, so it follows that
	\begin{equation}
		\left|g^G(\beta,\psi,T)\right|^2=\left|U\right|. \label{eq:gauss-sum-mag}
	\end{equation}
	Thus, the sum in the definition of $g^G(\beta,\psi,T)$ obeys the expected ``square root'' cancellation generically. (When $G\in\{\GO_{2n},\O_{2n}\}$, it may appear that our signs may disagree, but recall from \Cref{rem:bad-orthogonal} that the statement is vacuous for odd $n$.)
    However, note that \eqref{eq:gauss-sum-mag} cannot hold when $\beta^2=1$ because $\left|g^G(\beta,\psi,T)\right|^2$ is (up to sign) an eigenvalue of $I\circ I$, but in general there need not be such an eigenvalue when $\beta^2=1$. We will compute the correct factor in \Cref{sec:gsum}.
\end{remark}
\begin{remark} \label{rem:generic-eigenvalue}
    When applicable, the above construction produces many linearly independent eigenvectors for our operator $I$. Indeed, each available $T$ produces a new eigenvector, and one can estimate that
    \[\left|\left\{T\in\op{GL}_n(\FF_q):\begin{bmatrix}
        1 & T \\ & 1
    \end{bmatrix}\in G\right\}\right|\]
    is approximately $\left|U\right|=\left|P\backslash P\eta_nP\right|$, which is approximately
    $\left|P\backslash G\right|=\dim\operatorname{Ind}_P^G\chi$. (In each inequality, the first set is Zariski open in the second set.) One can follow these estimates to see that the eigenspace of $I$ with eigenvalue given by the Gauss sum is large.
\end{remark}

%% file: intertwining/gsum.tex

\section{Computation of Matrix Gauss Sums} \label{sec:gsum}
As before, let $\FF_q$ denote the finite field with $q$ elements, where $q$ is an odd prime power. For characters $\omega\colon\FF_q^\times\to\CC^\times$ and $\psi\colon\FF_q\to\CC$, we are interested in computing sums of the form
\[\sum_A\omega(\det A)\psi(\tr AT),\]
where $A$ and $T$ are possibly subject to certain constraints (e.g., symmetric or alternating). To be explicit, our sums will be done for the following three cases:
\begin{itemize}
    \item $\GL_n(\FF_q)$.
    \item $\Sym_n^\times(\FF_q)$, the set of invertible $n\times n$ symmetric matrices with coefficients in $\FF_q$.
    \item $\Alt_{2n}^\times(\FF_q)$, the set of invertible $2n\times 2n$ alternating matrices with coefficients in $\FF_q$. (Note that there are no invertible alternating matrices of odd dimension.)
\end{itemize}
For brevity, we will remove $\FF_q$ from our notation as much as possible.

Note that the sum over $A\in\GL_n$ has already been considered by \cite{kim-gauss-sum} and many authors before; see \cite[Section 1]{kim-gauss-sum}. Additionally, the sum over symmetric matrices was considered in \cite{saito-sym-gauss-sum}; their method is based on a rather lengthy computation with the Bruhat decomposition. To the authors' knowledge, the sum over alternating matrices has not been previously evaluated.

Our method is rather uniform over all kinds of sums considered. We will induct on the size of $A$ via an explicit row-reduction. As such, the arguments are essentially the same as the spirit of the arguments in \cite{kim-gauss-sum} in the case of $\GL_n$. However, we believe that it is expositionally justified to include our proofs here because they motivate the rest of the section.

\subsection{Miscellaneous Computations}
We take a moment to discuss a few sums which will be used frequently in the sequel. For characters $\omega\colon\FF_q^\times\to\CC^\times$ and $\psi\colon\FF_q\to\CC^\times$, we denote the usual Gauss sum by
\[g(\omega,\psi)\coloneqq\sum_{a\in\FF_q^\times}\omega(a)\psi(a).\]
We will utilize several identities for Gauss sums, which we lay out now. The first is a well-known identity for quadratic Gauss sums.
\begin{prop} \label{prop:mag-gauss-sum}
    Let $\omega\colon\FF_q^\times\to\CC^\times$ and $\psi\colon\FF_q\to\CC^\times$ denote nontrivial characters. Then
    \[g(\omega,\psi)g\left(\omega^{-1},\psi^{-1}\right)=q.\]
    Thus, if $\chi\colon\FF_q^\times\to\CC^\times$ denotes the nontrivial quadratic character, then $g(\chi,\psi)^2=\chi(-1)q$.
\end{prop}
\begin{proof}
    These are standard facts about Gauss sums.
    %
\end{proof}
The computation of the Gauss sums over $\Sym_n^\times$ will use the following fact.
\begin{prop} \label{prop:quad-twist-gauss-sum}
    Let $\omega\colon\FF_q^\times\to\CC^\times$ and $\psi\colon\FF_q\to\CC^\times$ be characters, and let $\chi\colon\FF_q^\times\to\CC^\times$ denote the nontrivial quadratic character. Then
    \[\omega(4)g(\omega,\psi)g(\omega\chi,\psi)=g\left(\omega^2,\psi\right)g(\chi,\psi).\]
\end{prop}
\begin{proof}
    Expanding out the Gauss sums, we are trying to show that
    \[\sum_{a,b\in\FF_q^\times}\omega(4ab)\chi(b)\psi(a+b)\stackrel?=\sum_{a,b\in\FF_q^\times}\omega\left(a^2\right)\chi(b)\psi(a+b).\]
    Fixing some $d\in\FF_q^\times$ and $t\in\FF_q$, it is enough to show that
    \begin{equation}
        \sum_{\substack{a+b=t\\4ab=d}}\chi(b) \stackrel?= \sum_{\substack{a+b=t\\a^2=d}}\chi(b) \label{eq:combo-quad-twist-gauss-sum}
    \end{equation}
    and then sum over all possible values of $d$ and $t$. At this point, the proof has become combinatorial number theory. For convenience, extend $\chi$ to $\FF_q$ by $\chi(0)\coloneqq0$, and allow $a,b\in\FF_q$ in the right-hand sum above; this will not change its value. We begin with some reductions.
    \begin{itemize}
        \item Suppose that $d$ is not a square. Then the right-hand side of \eqref{eq:combo-quad-twist-gauss-sum} is empty and hence zero. We claim that the left-hand side is zero. For any $(a,b)$ solving $a+b=t$ and $4ab=d$, we see that $(b,a)$ also solves the system. Then because $d$ is not a square, we have $\{\chi(a),\chi(b)\}=\{+1,-1\}$, so the terms in the sum cancel out.
        \item In the rest of the proof, we may assume that $d=x^2$ where $x\in\FF_q^\times$, so the right-hand side of \eqref{eq:combo-quad-twist-gauss-sum} is $\chi(t+x)+\chi(t-x)$.
        
        To continue, observe that solving the system of equations $a+b=t$ and $4ab=d$ is equivalent to having $a=t-b$ and
        \[(2b-t)^2=t^2-d.\]
        As such, for our next case, suppose that $t^2-d$ fails to be a square. Then the left-hand side of \eqref{eq:combo-quad-twist-gauss-sum} is empty and hence vanishes, so we want to show that the right-hand side also vanishes. But if $t^2-d=(t+x)(t-x)$ is non-square, then $\{\chi(t+x),\chi(t-x)\}=\{+1,-1\}$, so the summands on the right-hand side cancel out.
        \item In the rest of the proof, we may assume that $t^2-d=y^2$ for some $y\in\FF_q$. For example, if $y=0$, then the left-hand side equals $\chi(t/2)$, and the right-hand side equals $\chi(2t)$.
    \end{itemize}
    At the current point, we can now say that $t^2=x^2+y^2$ where $x,y\in\FF_q^\times$. Every solution $(a,b)$ to the system of equations on the left can be written as $\left(\frac{t+y}{2}, \frac{t-y}{2}\right)$, while those for the system of equations on the right can be written as $(t+x,t-x)$. Thus, to show \eqref{eq:combo-quad-twist-gauss-sum}, it is enough to show
    \begin{equation*}
        \chi\left(\frac{t+y}2\right)+\chi\left(\frac{t-y}2\right)\stackrel?=\chi(t+x)+\chi(t-x). \label{eq:last-combo-quad-twist-case}
    \end{equation*}
    Because $(t-x)(t+x)=y^2$ and $\left(\frac{t+y}2\right)\left(\frac{t-y}2\right)=\frac14x^2$, we see $\chi\left(\frac{t+y}2\right)=\chi\left(\frac{t-y}2\right)$ and $\chi(t+x)=\chi(t-x)$. As these values are in $\{\pm1\}$, it is enough to show that $\chi(t+x)=1$ if and only if $\chi\left(\frac{t+y}2\right)=1$. We will show the forward implication; the reverse implication is similar.

    Suppose $\chi(t+x) = 1$. As both $t+x$ and $t-x$ are squares, we can write $t+x=x_1^2$ and $t-x=x_2^2$ for $x_1,x_2\in\FF_q^\times$. Adjusting signs, we may assume that $y=x_1x_2$. Thus,
    \[\frac{t+y}2=\left(\frac{x_1+x_2}2\right)^2\]
    is a square, hence $\chi\left(\frac{t+y}{2}\right) = 1$ as desired.
    %
\end{proof}
All of our computations will frequently sum over vectors in some way, so we pick up the following fact.
\begin{lemma} \label{lem:matrix-char-sum}
    Fix a character $\psi\colon\FF_q\to\CC^\times$ and some $A\in\FF_q^{n\times m}$. Then
    \[\sum_{B\in\FF_q^{m\times n}}\psi(\tr AB)=\begin{cases}
        0 & \text{if }A\ne0\text{ and }\psi\ne1, \\
        q^{mn} & \text{if }A=0\text{ or }\psi=1.
    \end{cases}\]
\end{lemma}
\begin{proof}
    Note that
    \[\tr AB=\sum_{i=1}^m\sum_{j=1}^nA_{ji}B_{ij},\]
    so
    \[\sum_{B\in\FF_q^{m\times n}}\psi(\tr AB)=\prod_{i=1}^m\prod_{j=1}^n\sum_{B_{ij} \in \FF_q}\psi(A_{ji}B_{ij}).\]
    If $A=0$ or $\psi=1$, then all summands are $1$, so we total to $q^{mn}$. Otherwise, say $A_{ji}\ne0$ for some given $(i,j)$. Then the factor $\sum_{B_{ij}}\psi(A_{ji}B_{ij})$ in the product will vanish, as desired.
\end{proof}

\subsection{The Sum Over \texorpdfstring{$\GL_n$}{GL}}
For the purposes of this subsection, we define
\[g_n(\omega,\psi,T)\coloneqq\sum_{A\in\GL_n}\omega(\det A)\psi(\tr AT)\]
where $\omega\colon\FF_q^\times\to\CC^\times$ and $\psi\colon\FF_q\to\CC^\times$ are characters, and $T\in\GL_n$. Even though our method to compute $g_n(\omega,\psi,T)$ is essentially equivalent to the one presented in \cite{kim-gauss-sum}, we present it here because it provides a reasonable background to the approach.

The following general results will be helpful.
\begin{lemma} \label{lem:gsum-gl-basic}
    Fix characters $\omega\colon\FF_q^\times\to\CC^\times$ and $\psi\colon\FF_q\to\CC^\times$ and some $T\in\GL_n$.
    \begin{listalph}
        \item For any $g,h\in\GL_n$, we have
        \[g_n(\omega,\psi,gTh)=\omega(\det gh)^{-1}g_n(\omega,\psi,T).\]
        \item If $\psi=1$, then $g_n(\omega,\psi,T)=0$ unless $\omega=1$.
    \end{listalph}
\end{lemma}
\begin{proof}
    Here, (a) follows from some quick rearranging, and (b) follows because $g_n(\omega,\psi,T)$ is the sum of the character $\omega$ on the group $\GL_n$.
\end{proof}
We wish to express the matrix Gauss sums $g_n(\omega, \psi, T)$ in terms of the standard Gauss sums $g(\omega,\psi)$. We do this by relating $g_n$ to $g_{n-1}$ via row-reduction. We split the Gauss sum $g_n$ into two cases: $A_{nn}\ne0$ and $A_{nn}=0$. We treat the case $A_{nn}\neq0$ first.
\begin{lemma} \label{lem:gsum-gl-not-0}
    Fix characters $\omega\colon\FF_q^\times\to\CC^\times$ and $\psi\colon\FF_q\to\CC^\times$. If $\psi\ne1$, then
    \[\sum_{\substack{A\in\GL_{n+1}\\A_{n+1,n+1}\ne0}}\omega(\det A)\psi(\tr A)=q^ng(\omega,\psi)g_n(\omega,\psi,1_n).\]
\end{lemma}
\begin{proof}
    The main idea is that
    \[\arraycolsep=1.4pt\begin{array}{ccccccccc}
        \GL_n &\times& \FF_q^n &\times& \FF_q^n &\times& \FF_q^\times &\to& \GL_{n+1} \\
        (B &,& v &,& w &,& c) &\mapsto& \begin{bsmallmatrix}
            1_n & v \\
            & 1
        \end{bsmallmatrix}\begin{bsmallmatrix}
            B \\ & c
        \end{bsmallmatrix}\begin{bsmallmatrix}
            1_n \\
            w^\intercal & 1
        \end{bsmallmatrix}
    \end{array}\]
    is a bijection onto elements of $A\in\GL_{n+1}$ with nonzero entry $A_{n+1,n+1}$.
    Our sum becomes
    \[\underbrace{\sum_{B\in\GL_n}\omega(\det B)\psi(\tr B)}_{g_n(\omega,\psi,1_n)}\sum_{c,v,w}\omega(c)\psi(c)\psi(\tr cvw^\intercal).\]
    With $\psi\ne1$, we can get cancellation by summing over $w$
    by \Cref{lem:matrix-char-sum}. In particular, we only get a nonzero contribution when $v=0$, leaving us with $q^ng_n(\omega,\psi,1_n)g(\omega,\psi)$
    after summing over $c$ as well.
\end{proof}
We now handle $A_{nn}=0$. This requires more care because the rightmost column must have a nonzero entry.
\begin{lemma} \label{lem:gsum-gl-0}
    Fix characters $\omega\colon\FF_q^\times\to\CC^\times$ and $\psi\colon\FF_q\to\CC^\times$. Suppose $\psi \neq 1$. Choosing some nonzero $\ov v,\ov w\in\FF_q^{n+2}$ such that $\ov v_{n+2}=\ov w_{n+2}=0$, we have
    \[\sum_{\substack{A\in\GL_{n+2}\\Ae_{n+2}=\ov v\\A^\intercal e_{n+2}=\ov w}}\omega(\det A)\psi(\tr A)=0.\]
\end{lemma}
\begin{proof}
    We will transform our sum into
    \[\sum_{\substack{A\in\GL_{n+2}\\Ae_{n+2}=\ov v\\A^\intercal e_{n+2}=\ov w}}\omega(\det A)\psi(\tr AT),\]
    where $\ov v_{n+1},\ov w_{n+1}\ne0$, at the cost of allowing $T$ to be specified a permutation matrix.
    We may rearrange the coordinates of $\ov v$ and $\ov w$ so that $\ov v_{n+1},\ov w_{n+1}\ne0$ by replacing $A$ with $\tau A\sigma$ for suitable permutation matrices $\tau$ and $\sigma$.
    Looking at the original sum, this does not change $\det A$ (one can add a sign to $\ov v$ or $\ov w$ if necessary), but it transforms $\tr A$ into $\tr A\sigma\tau$. For brevity, we set $T\coloneqq\sigma\tau$. Note the construction promises that $\sigma(n+2)=n+2$.

    The rest of the argument proceeds as before. Write $\ov v=(cv,c,0)^\intercal$ and $\ov w=(dw,d,0)^\intercal$. Then, the point is that
    \[\arraycolsep=1.4pt\begin{array}{ccccccccccccc}
        \GL_n &\times& \FF_q^n &\times& \FF_q^n &\times& \FF_q &\to& \GL_{n+2} \\
        (B &,& v' &,& w' &,& e) &\mapsto& \begin{bsmallmatrix}
            1_n & v & v' \\
            & 1 \\ && 1
        \end{bsmallmatrix}\begin{bsmallmatrix}
            B \\ & e & c \\ & d
        \end{bsmallmatrix}\begin{bsmallmatrix}
            1_n \\
            w^\intercal & 1 \\ (w')^\intercal && 1
        \end{bsmallmatrix}
    \end{array}\]
    is a bijection onto $A\in\GL_{n+2}$ satisfying $Ae_{n+2}=\ov v$ and $A^\intercal e_{n+2}=\ov w$.
    
    We now show that the sum vanishes, splitting into two cases for the permutation matrix $\sigma$.
    \begin{itemize}
        \item Suppose that $\sigma(n+1)=n+1$. Then we may write $\sigma$ as $\begin{bsmallmatrix}
            T_n \\ & 1_2
        \end{bsmallmatrix}$. Then our sum looks like
        \begin{align*}
            & \sum_{B\in\GL_n}\omega(\det B)\psi(\tr BT_n) \\
            \times{}& \sum_{v',w',e}\omega(-cd)\psi(\tr cv(w')^\intercal T_n)\psi(\tr dv'w^\intercal T_n)\psi(\tr ecdvw^\intercal T_n)\psi(e).
        \end{align*}
        By \Cref{lem:matrix-char-sum}, we see that the sum over $w'$ will only produce nonzero contribution if $v=0$. But in this case, the sum over $e$ is just $\sum\psi(e)=0$, so the total sum vanishes.
        
        \item Suppose $\sigma(n+1)\ne n+1$; say $\sigma(i_0)=n+1$ for $i_0<n+1$. Here, we sum over $v'$ while holding all other variables fixed. The determinant does not depend on $v'$, so we are left summing over the $\psi$ terms. Only paying attention to $v'$, we see that we are computing
        \[\sum_{v'\in\FF_q^n}\prod_{i=1}^{n+2}\psi\left(e_i^\intercal \begin{bmatrix}
            dv'w^\intercal & dv' & 0 \\
            0 & 0 & 0 \\
            0 & 0 & 0
        \end{bmatrix}e_{\sigma(i)}\right).\]
        We now sum over $v'_{i_0}$ and hold the remaining coordinates $v'_\bullet$ constant. Then the only non-constant factor in the product is $i=i_0$, where $\sigma(i)=n+1$, thus producing the sum $\sum\psi(dv'_{i_0})=0$.
        \qedhere
    \end{itemize}
\end{proof}
We now synthesize our cases to evaluate our Gauss sums.
\begin{theorem} \label{thm:gsum-gl}
    Fix characters $\omega\colon\FF_q^\times\to\CC^\times$ and $\psi\colon\FF_q\to\CC^\times$ and some $T\in\GL_n$.
    \begin{listalph}
        \item Suppose $\psi\ne1$. Then
        \[g_n(\omega,\psi,T)=\frac{q^{n(n-1)/2}}{\omega(\det T)}\cdot g(\omega,\psi)^n.\]
        \item Suppose $\psi=1$ and $\omega=1$. Then
        \[g_n(\omega,\psi,T)=\prod_{i=0}^{n-1}\left(q^n-q^i\right).\]
    \end{listalph}For any $(\psi,\omega)$ not in the above list, the sum vanishes when $n$ is positive
\end{theorem}
\begin{proof}
    Note the last sentence follows by \Cref{lem:gsum-gl-basic}. Note both (a) and (b) reduce to the case where $T=1_n$ by \Cref{lem:gsum-gl-basic} because both sides are invariant under replacing $T$ by $gT$ for some $g\in\GL_n$. Now, the sum in (b) is simply enumerating $\GL_n$, so the result follows from \cite[Proposition~7.1.1]{hach-gf}.

    It remains to show (a). For $n=0$, there is nothing to prove. Thus, by induction, it is enough to show that
    \[g_{n+1}(\omega,\psi,1_{n+1})\stackrel?=q^ng(\omega,\psi)g_n(\omega,\psi,1_n)\]
    for $n\ge0$, which follows by summing \Cref{lem:gsum-gl-not-0,lem:gsum-gl-0}.
\end{proof}
\begin{remark}
    It is possible to prove (b) in the theorem by tracking the case $(\omega,\psi)=(1,1)$ through \Cref{lem:gsum-gl-not-0,lem:gsum-gl-0}. We have not done so for brevity.
\end{remark}
We close this subsection with a combinatorial application; note there is a similar result in \cite[Theorem~6.2]{kim-gauss-sum}.
\begin{corollary}
    Let $n$ be a nonnegative integer, and fix some $T\in\GL_n$. Further, fix $d\in\FF_q^\times$ and $t\in\FF_q$. Then the number $N(d,t)$ of $A\in\GL_n$ such that $\det A=d$ and $\tr AT=t$ is
    \begin{align*}
        &\frac1{q(q-1)}\left(\prod_{i=0}^{n-1}\left(q^n-q^i\right)-q^{n(n-1)/2}(q-1)^n\right) \\
        &+q^{n(n-1)/2}\cdot\#\left\{(y_1,\ldots,y_n)\in\FF_q^n:(y_1+\cdots+y_n)=t,\frac{y_1\cdots y_n}{\det T}=d\right\}.
    \end{align*}
\end{corollary}
\begin{proof}
    For any characters $\omega\colon\FF_q^\times\to\CC^\times$ and $\psi\colon\FF_q\to\CC^\times$, we claim that $g_n(\omega,\psi,T)$ equals
    \begin{align*}
        &\frac1{q(q-1)}\left(\prod_{i=0}^{n-1}\left(q^n-q^i\right)-q^{n(n-1)/2}(q-1)^n\right)\sum_{a\in\FF_q^\times,b\in\FF_q}\omega(a)\psi(b) \\
        &+ \frac{q^{n(n-1)/2}}{\omega(\det T)}\cdot g(\omega,\psi)^n
    \end{align*}
    If $\psi\ne1$, then this is (a) of \Cref{thm:gsum-gl}; if $\psi=1$, then both sides vanish unless $\omega=1$, in which case this is (b) of \Cref{thm:gsum-gl}. Now, we notice that full expansion gives
    \[\frac1{\omega(\det T)}\cdot g(\omega,\psi)^n=\sum_{y_1,\ldots,y_n\in\FF_q^\times}\omega\left(\frac{y_1\cdots y_n}{\det T}\right)\psi(y_1+\cdots+y_n),\]
    so the result follows by summing appropriately over all $\omega$ and $\psi$.
\end{proof}

\subsection{The Sum Over \texorpdfstring{$\Sym_n^\times$}{ Sym}}
For the purposes of this subsection, we define
\[g_n(\omega,\psi,T)\coloneqq\sum_{A\in\Sym_n^\times}\omega(\det A)\psi(\tr AT)\]
where $\omega\colon\FF_q^\times\to\CC^\times$ and $\psi\colon\FF_q\to\CC^\times$ are characters, and $T\in\Sym_n^\times$. Additionally, throughout we let $\chi\colon\FF_q^\times\to\CC^\times$ denote the nontrivial quadratic character.

We follow the arguments of the previous subsection on $\GL_n$.
\begin{lemma} \label{lem:gsum-sym-basic}
    Fix characters $\omega\colon\FF_q^\times\to\CC^\times$ and $\psi\colon\FF_q\to\CC^\times$ and some $T\in\Sym_n^\times$.
    \begin{listalph}
        \item For any $g\in\GL_n$, we have
        \[g_n(\omega,\psi,gTg^\intercal)=\omega(\det g)^{-2}g_n(\omega,\psi,T).\]
        \item If $\psi=1$ and $\omega\ne1$, then $g_n(\omega,\psi,T)=0$ unless $\omega^2=1$ and $n$ is even.
    \end{listalph}
\end{lemma}
\begin{proof}
    Here, (a) follows by some elementary rearrangement upon noticing that $\op{GL}_n$ acts on $\op{Sym}_n^\times$ by $g\cdot A=gAg^\intercal$. For (b), we handle the two listed cases separately.
    \begin{itemize}
        \item Suppose $\omega^2\ne1$. Then the rearrangement which shows (a) is able to show that $g_n(\omega,1,T)=\omega(\det g)^2g_n(\omega,1,T)$ for any $g\in\GL_n$ (because $\psi=1$), so the $g_n(\omega,1,T)=0$ follows.
        \item Suppose $n$ is odd and $\omega^2=1$. Now, for any $c\in\FF_q^\times$, we see that $A\in\Sym_n^\times$ if and only if $cA\in\Sym_n^\times$, so some rearrangement shows $g_n(\omega,1,T)=\omega(c)^ng_n(\omega,1,T)$ for any $c\in\FF_q^\times$. However, $n$ is odd, so $\omega^n=\omega$ is nontrivial, so this forces $g_n(\omega,1,T)=0$.
        \qedhere
    \end{itemize}
\end{proof}
As before, our row-reduction will have two cases: $A_{nn}\ne0$ and $A_{nn}=0$.
\begin{lemma} \label{lem:gsum-sym-ind-not-0}
    Fix characters $\omega\colon\FF_q^\times\to\CC^\times$ and $\psi\colon\FF_q\to\CC^\times$ and diagonal $T_{n+1}\in\Sym_{n+1}^\times$. Letting $T_n\coloneqq\op{diag}(T_{11},\ldots,T_{nn})$, if $\psi\ne1$, then the sum
    \[\sum_{\substack{A\in\Sym_{n+1}^\times\\A_{n+1,n+1}\ne0}}\omega(\det A)\psi(\tr AT_{n+1})\]
    equals
    \[g_n(\omega,\psi,T_n)\frac{\chi(\det T_n)\chi(T_{n+1,n+1})^n}{\omega(T_{n+1,n+1})}g(\omega\chi^n,\psi)g(\chi,\psi)^n.
    \]
\end{lemma}
\begin{proof}
    The main point is that
    \[\arraycolsep=1.4pt\begin{array}{ccccccc}
        \operatorname{Sym}_n^\times & \times & \FF_q^n & \times & \FF_q^\times & \to & \Sym_{n+1}^\times \\
        (B & , & v & , & c) & \mapsto & \displaystyle\begin{bsmallmatrix}
            1 & v \\
              & 1
        \end{bsmallmatrix}\begin{bsmallmatrix}
            B \\ & c
        \end{bsmallmatrix}\begin{bsmallmatrix}
            1 \\ v^\intercal & 1
        \end{bsmallmatrix}
    \end{array}\]
    is a bijection onto $A\in\Sym_{n+1}^\times$ with $A_{n+1,n+1}\ne0$. 
    Thus, our sum is
    \[\underbrace{\sum_{B\in\Sym_n^\times}\omega(\det B)\psi(\tr BT_n)}_{g_n(\omega,\psi,T_n)}\sum_{v,c}\omega(c)\psi(\tr cvv^\intercal T_n)\psi(cT_{n+1,n+1}).\]
    %
    Now, for brevity, we set $T\coloneqq\op{diag}(d_1,\ldots,d_{n+1})$, so the sum over $v$ and $c$ above equals
    \[\sum_{c\in\FF_q^\times}\omega(c)\psi(cd_{n+1})\prod_{i=1}^n\Bigg(\sum_{a\in\FF_q}\psi\left(cd_ia^2\right)\Bigg)\]
    after some expansion (of $v\in\FF_q^n$). Quickly, we note that
    \[\sum_{a\in\FF_q}\psi\left(cd_ka^2\right)\stackrel?=\sum_{a\in\FF_q}(1+\chi(cd_ka))\psi(a),\]
    where we have extended $\chi$ to $\FF_q$ by $\chi(0)\coloneqq0$; indeed, $(1+\chi(cd_ka))$ is an appropriate indicator for $a/(cd_k)$ being a square. From here, we note $\psi\ne1$ implies
    \[\sum_{a\in\FF_q}\psi\left(cd_ka^2\right)=\chi(cd_k)g(\chi,\psi).\]
    Plugging this in, we see that our sum is
    \[g_n(\omega,\psi,T_n)\sum_{c\in\FF_q^\times}\omega(c)\chi(c)^n\psi(cd_{n+1})\chi(d_1\cdots d_n)g(\chi,\psi)^n,\]
    which rearranges into the desired.
\end{proof}
Next, we handle $A_{nn}=0$.
\begin{lemma} \label{lem:gsum-sym-ind-0}
    Fix characters $\omega\colon\FF_q^\times\to\CC^\times$ and $\psi\colon\FF_q\to\CC^\times$ and some diagonal $T_{n+2}\in\Sym^\times_{n+2}$, and let $T_n\coloneqq\op{diag}(T_{11},\ldots,T_{nn})$. Choosing some nonzero $\ov v\in\FF_q^{n+2}$ such that $\ov v_{n+2}=0$, if $\psi\ne1$, we have
    \[\sum_{\substack{A\in\Sym_{n+2}^\times\\Ae_{n+2}=\ov v}}\omega(\det A)\psi(\tr AT_{n+2})=0.
    \]
\end{lemma}
\begin{proof}
    We begin by reducing to the case $\ov v_{n+1}\ne0$. Set $T\coloneqq T_{n+2}$ for brevity. We want to rearrange the coordinates of $\ov v$ so that $v_{n+1}\ne0$ by mapping $A\mapsto\sigma^{-1}A\sigma$ for suitable permutation matrix $\sigma$. This does not change $\det A$, but it transforms $\tr AT$ into $\tr A\sigma T\sigma^{-1}$, effectively rearranging the rows and columns of $T$ into a different diagonal matrix. However, the conclusion is independent of $T$, so this rearrangement is legal.

    We may now row-reduce. Write $\ov v=(cv,c,0)$ (as a column vector). The main point is that there is a bijection
    \[\arraycolsep=1.4pt\begin{array}{ccccccc}
        \operatorname{Sym}_{n}^\times & \times & \FF_q^{n} & \times & \FF_q & \to & \Sym_{n+2}^\times \\
        (B & , & w & , & d) & \mapsto &  \begin{bsmallmatrix}
            1 & v & w \\
                & 1 &   \\
                &   & 1
        \end{bsmallmatrix}\begin{bsmallmatrix}
            B \\
                & d & c \\
                & c
        \end{bsmallmatrix}\begin{bsmallmatrix}
            1 \\
            v^\intercal & 1 \\
            w^\intercal &   & 1
        \end{bsmallmatrix}
    \end{array}\]
    onto the set of $A\in\Sym_{n+2}^\times$ such that $Ae_{n+2}=\ov v$.
    Thus, we see that our sum is
    \begin{align*}
        &\sum_{B}\omega\left(-c^2\det B\right)\psi(\tr BT_n) \\
        \times{}& \sum_{d}\psi\left(dT_{n+1,n+1}+d\tr vv^\intercal T_n\right)\sum_{w}\psi(2c\tr vw^\intercal T_n).
    \end{align*}
    With $\psi\ne1$, we see that the sum over $w$ vanishes by \Cref{lem:matrix-char-sum} unless $v=0$. But in the case where $v=0$, we see that the sum over $d$ will vanish, so the total sum continues to vanish.
\end{proof}
We now synthesize our cases to evaluate our Gauss sums.
\begin{theorem} \label{thm:gsum-sym}
    Fix characters $\omega\colon\FF_q^\times\to\CC^\times$ and $\psi\colon\FF_q\to\CC^\times$ and some $T\in\Sym_n^\times$. Let $\chi\colon\FF_q^\times\to\CC^\times$ denote the quadratic character.
    \begin{listalph}
        \item Suppose $\psi\ne1$.
        \begin{itemize}
            \item If $n=2m$ is a positive even integer, then
            \[g_{2m}(\omega,\psi,T)=\frac{\chi(-1)^m\chi(\det T)q^{m^2}}{\omega(4^m\det T)}\cdot g\left(\omega^2,\psi\right)^m.\]
            \item If $n=2m+1$ is an odd nonnegative integer, then
            \[g_{2m+1}(\omega,\psi,T)=\frac{q^{m(m+1)}}{\omega(4^m\det T)}\cdot g(\omega,\psi)g\left(\omega^2,\psi\right)^m.\]
        \end{itemize}
        \item Suppose $\psi=1$ and $\omega=\chi$. If $n=2m$ is even, then
        \[g_{2m}(\chi,1,T)=\chi(-1)^mq^{m^2}\prod_{k=0}^{m-1}\left(q^{2k+1}-1\right).\]
        \item Suppose $\psi=1$ and $\omega=1$.
        \begin{itemize}
            \item If $n=2m$ is even, then
            \[g_{2m}(1,1,T)=q^{m^2+m}\prod_{k=0}^{m-1}\left(q^{2k+1}-1\right).\]
            \item If $n=2m+1$ is odd, then
            \[g_{2m+1}(1,1,T)=q^{m^2+m}\prod_{k=0}^{m}\left(q^{2k+1}-1\right).\]
        \end{itemize}
    \end{listalph}
    For any $(\psi,\omega)$ not in the above list, the sum vanishes when $n$ is positive
\end{theorem}
\begin{proof}
    Note the last sentence follows by \Cref{lem:gsum-sym-basic}. Additionally, we may quickly handle (b) and (c), where $\psi=1$ because they reduce to the combinatorics of symmetric matrices: (c) is counting invertible symmetric matrices, so the result is \cite[Theorem~7.5.2]{hach-gf}. Similarly, (b) is counting the difference between invertible symmetric matrices with square and non-square determinant, which is computed in \cite[p.~163]{macwilliams-ortho-matrices}.
    
    It remains to prove (a). We proceed by indution on $n$, where the case of $n=1$ can be checked directly. We quickly reduce to the case where $T$ is diagonal. By choosing an orthogonal basis for the symmetric bilinear form given by $T$, we receive some $g\in\GL_n$ such that $D\coloneqq gTg^\intercal$ is diagonal. As such, \Cref{lem:gsum-sym-basic} yields
    \[g_n(\omega,\psi,T) = \omega(\det g)^2g_n(\omega,\psi,D).\]
    Now, suppose we have proven the theorem for diagonal matrices. In this case, we see $g_n(\omega,\psi,D)=(\det D)^{-1}g_n(\omega,\psi,1)$, so $\det D=(\det g)^2(\det T)$ implies that
    \[g_n(\omega,\psi,T)=\omega(\det T)^{-1}g_n(\omega,\psi,1),\]
    which is the theorem for $T$, as desired.

    Thus, we may assume that $T_n\coloneqq\op{diag}(d_1,\ldots,d_n)$; set $T_{n-1}\coloneqq\op{diag}(d_1,\ldots,d_{n-1})$ and define $T_{n-2}$ analogously. We now induct on $n$ in cases.
    \begin{itemize}
        \item Suppose that $n=2m$ is an even positive integer. In this case, \Cref{lem:gsum-sym-ind-not-0,lem:gsum-sym-ind-0} and induction show $g_{2m}(\omega,\psi,T)$ equals
        \[\frac{\chi(\det T)q^{(m-1)m}}{\omega(4^{m-1}\det T)}\cdot g(\omega,\psi)g\left(\omega^2,\psi\right)^{m-1}g(\omega\chi,\psi)g(\chi,\psi)^{2m-1}.\]
        We now recall $g(\chi,\psi)^2=\chi(-1)q$ by \Cref{prop:mag-gauss-sum} and $\omega(4)g(\omega,\psi)g(\omega\chi,\psi)=g\left(\omega^2,\psi\right)g(\chi,\psi)$ by \Cref{prop:quad-twist-gauss-sum}, so rearrangement completes the proof.
        \item Suppose $n=2m+1$ is an odd positive integer with $m\ge1$. In this case, \Cref{lem:gsum-sym-ind-not-0,lem:gsum-sym-ind-0} and induction show $g_{2m+1}(\omega,\psi,T)$ equals
        \[\frac{\chi(-1)^mq^{m^2}}{\omega(4^m\det T)}\cdot g\left(\omega^2,\psi\right)^mg(\omega,\psi)g(\chi,\psi)^{2m}.\]
        We are now done after recalling $g(\chi,\psi)^2=\chi(-1)q$ by \Cref{prop:mag-gauss-sum} and rearranging.
        \qedhere
    \end{itemize}
\end{proof}
\begin{remark}
    It is possible to prove (b) and (c) above using the same inductive method which proves (a). 
\end{remark}
We conclude this subsection with a combinatorial application.
\begin{corollary} \label{cor:count-sym}
    Let $n$ be a nonnegative integer, and fix some $T\in\Sym_n^\times$. Further, fix $d\in\FF_q^\times$ and $t\in\FF_q$.
    \begin{enumerate}[label=(\alph*)]
        \item Suppose that $n=2m+1$ is odd. Then the number $N(d,t)$ of $A\in\Sym_{2m+1}^\times$ such that $\det A=d$ and $\tr AT=t$ is
        \begin{align*}
            &\frac{q^{m^2+m}}{q(q-1)}\Bigg(\prod_{k=0}^m\left(q^{2k+1}-1\right)-(q-1)^{m+1}\Bigg) \\
            &+q^{m^2+m}\#\left\{(y_0,\ldots,y_m)\in\FF_q^{m+1}:y_0+\cdots+y_m=t,\frac{y_0(y_1\cdots y_m)^2}{4^m\det T}=d\right\}.
        \end{align*}
        \item Suppose that $n=2m$ is even. Let $\chi\colon\FF_q^\times\to\CC^\times$ denote the nontrivial quadratic character. Then the number $N(d,t)$ of $A\in\Sym_{2m}^\times$ such that $\det A=d$ and $\tr AT=t$ is
        \begin{align*}
            &\frac{q^{m^2}}{q(q-1)}\Bigg(\left(q^m+\chi(-1)^m\chi(d)\right)\prod_{k=0}^{m-1}\left(q^{2k+1}-1\right) \\
            & \qquad\qquad\qquad-\chi(-1)^m\left(\chi(d)+\chi(\det T)\right)(q-1)^m\Bigg) \\
            &+\chi((-1)^m\det T)q^{m^2}\#\left\{(y_1,\ldots,y_m)\in\FF_q^m:y_1+\cdots+y_m=t,\frac{(y_1\cdots y_m)^2}{4^m\det T}=d\right\}.
        \end{align*}
    \end{enumerate}
\end{corollary}
\begin{proof}
    We prove these separately.
    \begin{enumerate}[label=(\alph*)]
        \item For any characters $\omega\colon\FF_q^\times\to\CC^\times$ and $\psi\colon\FF_q\to\CC^\times$, we claim that $g_n(\omega,\psi,T)$ equals
        \begin{align*}
            & \frac{g_n(1,1,T)-q^{m(m+1)}(q-1)^{m+1}}{q(q-1)}\sum_{a\in\FF_q^\times,b\in\FF_q}\omega(a)\psi(b) \\
            &+ \frac{q^{m(m+1)}}{\omega\left(4^m\det T\right)}\cdot g(\omega,\psi)g\left(\omega^2,\psi\right)^m
        \end{align*}
        This is by casework, using \Cref{thm:gsum-sym} repeatedly. If $\psi$ is nontrivial, the sum vanishes, so the claim follows from \Cref{thm:gsum-sym}. If $\psi$ is trivial and $\omega$ is nontrivial, then everything vanishes. Lastly, if both $\psi$ and $\omega$ are trivial, then everything equals $g_n(1,1,T)$.

        Now, we notice that fully expanding $g(\omega,\psi)g\left(\omega^2,\psi\right)^m$ gives
        \[
        \sum_{y_0,y_1,\ldots,y_m\in\FF_q^\times}\omega\left(y_0(y_1\cdots y_m)^2\right)\psi(y_0+\cdots+y_m),\]
        so we achieve the result by summing appropriately over all $\omega$ and $\psi$ and using the formula for $g_n(1,1,T)$ given in \Cref{thm:gsum-sym}.

        \item For any characters $\omega\colon\FF_q^\times\to\CC^\times$ and $\psi\colon\FF_q\to\CC^\times$, we claim that $g_n(\omega,\psi,T)$ equals
        \begin{align*}
            &\frac{\chi(-1)^m\chi(\det T)q^{m^2}}{\omega\left(4^m\det T\right)}\cdot g\left(\omega^2,\psi\right)^m \\
            &+\frac{g_n(\chi,1,T)-\chi(-1)^mq^{m^2}(q-1)^m}{q(q-1)}\sum_{a\in\FF_q^\times,b\in\FF_q}\chi(a)\omega(a)\psi(b) \\
            &+\frac{g_n(1,1,T)-\chi((-1)^m\det T)q^{m^2}(q-1)^m}{q(q-1)}\sum_{a\in\FF_q^\times,b\in\FF_q}\omega(a)\psi(b).
        \end{align*}
        Again, this is by casework, repeatedly using \Cref{thm:gsum-sym}. If $\psi$ is nontrivial, this is \Cref{thm:gsum-sym}. Otherwise, $\psi$ is trivial. Then if $\omega^2\ne1$, then $\omega\notin\{1,\chi\}$, so everything vanishes. Lastly, if $\omega\in\{1,\chi\}$, then both sides are equal by construction.

        The rest of the proof proceeds as in (a) by expanding out $g\left(\omega^2,\psi\right)^m$ and summing over $\omega$ and $\psi$ appropriately.
        \qedhere
    \end{enumerate}
\end{proof}

\subsection{The Sum Over \texorpdfstring{$\Alt_{2n}^\times$}{ Alt}}
For the purposes of this subsection, we define
\[g_{2n}(\omega,\psi,T)\coloneqq\sum_{A\in\Alt_{2n}^\times}\omega(\det A)\psi(\tr AT)\]
where $\omega\colon\FF_q^\times$ and $\psi\colon\FF_q\to\CC^\times$ are characters, and $T\in\Alt_{2n}^\times$. Additionally, throughout we let $J\coloneqq\begin{bsmallmatrix}
    & -1 \\ 1
\end{bsmallmatrix}$, which is the $2\times2$ version of the matrix defined in \Cref{sec:rep-theory}.

We would like to follow the outline established for $\GL_n$ and row-reduce, but a more complicated induction will be required because alternating matrices have vanishing diagonal. Instead, our row-reduction will be based on subdividing $A$ into $2\times2$ minors. As such, our casework is based on $A_{2n,2n-1}=-A_{2n-1,2n}$. Otherwise, our outline is the same.
\begin{lemma} \label{lem:gsum-alt-basic}
    Fix characters $\omega\colon\FF_q^\times$ and $\psi\colon\FF_q\to\CC^\times$ and some $T\in\Alt_{2n}^\times$.
    \begin{listalph}
        \item For any $g\in\GL_{2n}$, we have
        \[g_{2n}(\omega,\psi,gTg^\intercal)=\omega(\det g)^{-2}g_{2n}(\omega,\psi,T).\]
        \item If $\psi=1$, then $g_{2n}(\omega,\psi,T)=0$ unless $\omega^2=1$.
    \end{listalph}
\end{lemma}
\begin{proof}
    Here, (a) follows after some rearrangement using the action of $\GL_{2n}$ on $\op{Alt}_{2n}^\times$ by $g\cdot A\coloneqq gAg^\intercal$. For (b), we use this same rearrangement to show $g_{2n}(\omega,1,T)=\omega(\det g)^2g_{2n}(\omega,1,T)$ for any $g\in\GL_{2n}$, so the result follows.
\end{proof}
We now handle $A_{2n,2n-1}\ne0$.
\begin{lemma} \label{lem:gsum-alt-not-0}
    Fix characters $\omega\colon\FF_q^\times\to\CC^\times$ and $\psi\colon\FF_q\to\CC^\times$, and set $T_{2i}\coloneqq\op{diag}(J,\ldots,J)\in\Alt^\times_{2i}$ for each $i$. Then if $\psi\ne1$,
    \[\sum_{\substack{A\in\Alt_{2n+2}^\times\\A_{2n,2n-1}\ne0}}\omega(\det A)\psi(\tr AT_{2n+2})=q^{2n}g\left(\omega^2,\psi^2\right)g_{2n}(\omega,\psi,T_{2n}).
    \]
\end{lemma}
\begin{proof}
    The point is that the bottom-right $2\times2$ minor of our $A\in\Alt_{2n+2}^\times$ is invertible. Thus, the main point is that
    \[\arraycolsep=1.4pt\begin{array}{ccccccccccc}
        \Alt_{2n}^\times &\times& \FF_q^{2n\times2} &\times& \FF_q^\times &\to& \Alt_{2n+2}^\times \\
        (B &,& V &,& c) &\mapsto& \begin{bsmallmatrix}
            1_{2n} & V \\ & 1_2
        \end{bsmallmatrix}\begin{bsmallmatrix}
            B \\ & cJ
        \end{bsmallmatrix}\begin{bsmallmatrix}
            1_{2n} \\ V^\intercal & 1_2
        \end{bsmallmatrix}
    \end{array}\]
    is a bijection onto $A\in\Alt_{2n+2}^\times$ with $A_{2n+2,2n+1}\ne0$. Indeed, letting $V=\begin{bsmallmatrix}
        v & w
    \end{bsmallmatrix}$, we find our sum is
    \begin{align*}
        & \sum_{B}\omega\left(\det B\right)\psi(\tr BT_{2n}) \\
        \times{} & \sum_{c,v,w}\omega\left(c^2\right)\psi(-2c)\psi(\tr(cwv^\intercal-cvw^\intercal)T_{2n}).
    \end{align*}
    With $\psi\ne1$, we note $-\tr vw^\intercal T_{2n}=\tr wv^\intercal T_{2n}$, so the sum over $v$ and $w$ is
    \[\sum_{v,w\in\FF_q^{2n}}\psi(-2c\tr vw^\intercal T_{2n}).\]
    Fixing $v$ and summing over $w$, \Cref{lem:matrix-char-sum} tells us that we only get a nonzero contribution when $v=0$, where we see the sum will evaluate to $q^{2n}$. In this case, the desired sum compresses down to $q^{2n}g\left(\omega^2,\psi^2\right)g_{2n}(\omega,\psi,T_{2n})$, as required.
\end{proof}
Next, we handle $A_{2n-1,2n}=0$.
\begin{lemma} \label{lem:gsum-alt-0}
    Take $n\ge1$. Fix characters $\omega\colon\FF_q^\times$ and $\psi\colon\FF_q\to\CC^\times$, and set $T_{2i}\coloneqq\op{diag}(J,\ldots,J)\in\Alt^\times_{2i}$ for each $i$. Choosing some nonzero vector $\ov v\in\FF_q^{2n+2}$ such that $\ov v_{2n+1}=\ov v_{2n+2}=0$, if $\psi\ne1$, we have
    \[\sum_{\substack{A\in\Alt_{2n+2}^\times\\Ae_{n+2}=\ov v}}\omega(\det A)\psi(\tr AT_{2n+2})=0.
    \]
\end{lemma}
\begin{proof}
    This argument is similar to \Cref{lem:gsum-gl-0}, but we are more careful with the permutation matrix. We would like to reduce to a case where $\ov v_{2n+1}\ne0$ by adjusting $T$ appropriately. Because $\ov v$ is nonzero, we may find an index $i_0\notin\{2n+1,2n+2\}$ such that $\ov v_{i_0}$ is nonzero. By mapping $A\mapsto\sigma A\sigma$ where $\sigma$ is the permutation matrix associated to some permutation of the form $(2i-1,2j-1)(2i,2j)$, we see that the sum will not change (because $\det\sigma A\sigma=\det A$ and $\sigma T_{2n+2}\sigma=T_{2n+2}$); thus, we may apply such a permutation to assume that $i_0\in\{2n-1,2n\}$. We now set $\sigma\coloneqq(i_0,2n+1)$ and apply $A\mapsto\sigma A\sigma$ to our sum, which does adjust $\ov v$ (so that $\ov v_{2n+1}\ne0$) as well as make our sum change $T_{2n+2}$ to the matrix $\sigma T_{2n+2}\sigma$, which is in
    \[\left\{\begin{bmatrix}
        T_{2n-2} \\
        &&&& -1 \\ &&& 1 \\ && -1 \\ & 1
    \end{bmatrix},\begin{bmatrix}
        T_{2n-2} \\
        &&& -1 \\ &&&& -1 \\ & 1 \\ && 1
    \end{bmatrix}\right\}\]
    (The left happens when $i_0=2n-1$, and the right happens when $i_0=2n$.) With our now adjusted $\ov v$, we write $\ov v=(-cv,-c,0)$ where $v\in\FF_q^{2n}$ and $c\in\FF_q^\times$, and we note we want to compute
    \[\sum_{\substack{A\in\Alt_{2n+2}^\times\\Ae_{n+2}=\ov v}}\omega(\det A)\psi(\tr A\sigma T_{2n+2}\sigma).\]
    Now, the main point is that
    \[\arraycolsep=1.4pt\begin{array}{cccccc}
        \Alt_{2n}^\times &\times& \FF_q^n &\to& \Alt_{2n+2}^\times \\
        (B &,& w) &\mapsto& \begin{bsmallmatrix}
            1_{2n} & v & w \\
            & 1 \\ && 1
        \end{bsmallmatrix}\begin{bsmallmatrix}
            B \\ && -c \\ & c
        \end{bsmallmatrix}\begin{bsmallmatrix}
            1_{2n} \\ v^\intercal & 1 \\ w^\intercal && 1
        \end{bsmallmatrix}
    \end{array}\]
    is a bijection onto $A\in\Alt_{2n+2}^\times$ with $Ae_{2n+2}=(-cv,-c,0)$.
    Now, to find cancellation in our sum, we hold $B$ constant and let $w$ vary.
    In particular, after some rearrangement, we see that the sum in question contains the factor
    \[\sum_{w\in\FF_q^n}\psi\left(\tr\begin{bmatrix}
        wv^\intercal-cvw^\intercal & cw & -cv \\
        -cw^\intercal && -c \\
        cv^\intercal & c
    \end{bmatrix}\sigma T_{2n+2}\sigma\right),\]
    which we will show vanishes. In fact, we will look at a factor of this sum. Because $\sigma T_{2n+2}\sigma$ takes the form $\op{diag}(T_{2n-2},T')$ for some $T'\in\Alt_4^\times$ described above, we see that the sum above contains the factor
    \[\sum_{w_{2n-1},w_{2n}}\psi\left(\tr\begin{bmatrix}
        & * & cw_{2n-1} & -cv_{2n-1} \\
        * & & cw_{2n} & -cv_{2n} \\
        -cw_{2n-1} & -cw_{2n} & & * \\
        cv_{2n-1} & cv_{2n} & *
    \end{bmatrix}T'\right)\]
    by using the bottom-right $4\times4$ minors of our matrices; here $*$s denote terms which do not matter. Now, if $i_0=2n-1$, then one finds that the sum over $w_{2n}$ vanishes; and if $i_0=2n$, then one finds that the sum over $w_{2n-1}$ vanishes.
\end{proof}
We now synthesize our cases.
\begin{theorem} \label{thm:gsum-alt}
    Fix characters $\omega\colon\FF_q^\times\to\CC^\times$ and $\psi\colon\FF_q\to\CC^\times$ and some $T\in\Alt_{2n}^\times$.
    \begin{listalph}
        \item Suppose $\psi\ne1$. Then
        \[g_{2n}(\omega,\psi,T)=\frac{q^{n(n-1)}}{\omega(\det T)}\cdot g\left(\omega^2,\psi^2\right)^n.\]
        \item Suppose $\psi=1$ and $\omega^2=1$. Then
        \[g_{2n}(\omega,1,T)=q^{n(n-1)}\prod_{i=1}^{n}\left(q^{2i-1}-1\right).\]
    \end{listalph}
    For any $(\psi,\omega)$ not in the above list, the sum vanishes when $n$ is positive.
\end{theorem}
\begin{proof}
    Note the last sentence follows from \Cref{lem:gsum-alt-basic}. We also quickly note that (b) is combinatorics: because any $A\in\op{Alt}_{2n}^\times$ can be written as $g\op{diag}(J,\ldots,J)g^\intercal$ for some $g\in\op{GL}_n$ (by finding a symplectic basis for $A$), we see that $\det A$ is a square always, so (b) is simply counting the number of invertible $2n\times2n$ alternating matrices. Thus, (b) follows from \cite[Theorem~7.5.5]{hach-gf}.
    
    It remains to show (a). We note we may reduce to the case where $T=\op{diag}(J,\ldots,J)$ using \Cref{lem:gsum-alt-basic}. We now induct on $n$, where the cases $n=0$ and $n=1$ can be checked directly. The induction now follows from summing \Cref{lem:gsum-alt-not-0,lem:gsum-alt-0}.
\end{proof}
\begin{remark}
    One can prove (b) via the same method as (a).
\end{remark}
Here is the corresponding combinatorial application.
\begin{corollary} \label{cor:count-alt}
    Fix some even nonnegative integer $2n$ and some $T\in\Alt_{2n}^\times$. Further, fix $d\in(\FF_q^{\times})^2$ and $t\in\FF_q$. Then the number $N(d,t)$ of $A\in\Alt_{2n}^\times$ such that $\det A=d$ and $\tr AT=t$ is
    \begin{align*}
        &\frac2{q(q-1)}\left(q^{n(n-1)}\prod_{i=1}^{n}\left(q^{2i-1}-1\right)-q^{n(n-1)}(q-1)^n\right) \\
        &+ q^{n(n-1)}\#\left\{(y_1,\ldots,y_n)\in\FF_q^n:2(y_1+\cdots+y_n)=t,\frac{(y_1\cdots y_n)^2}{\det T}=d\right\}.
    \end{align*}
\end{corollary}
\begin{proof}
    For any characters $\omega\colon\FF_q^\times$ and $\psi\colon\FF_q\to\CC^\times$, we claim that $g_{2n}(\omega,\psi,T)$ equals
    \begin{align*}
        &\frac 2{q(q-1)}
        \left(g_{2n}(1,1,T)-q^{n(n-1)}(q-1)^n\right)\sum_{a\in(\FF_q^{\times})^2,b\in\FF_q}\omega^2(a)\psi(b) \\
        &+\frac{q^{n(n-1)}}{\omega(\det T)}\cdot g\left(\omega^2,\psi^2\right)^n.
    \end{align*}
    This is the usual casework with \Cref{thm:gsum-alt}: if $\psi\ne1$, then the top row vanishes; if $\psi=1$ and $\omega^2\ne1$, then everything vanishes; and if $\psi=1$ and $\omega^2=1$, then this holds by construction.

    The result now follows by a direct expansion of $g\left(\omega^2,\psi^2\right)$
    as
    \[\sum_{y_1,\ldots,y_n}\omega\left((y_1\cdots y_n)^2\right)\psi(2(y_1+\cdots+y_n)),\]
    and then summing over $\omega$ and $\psi$ appropriately.
\end{proof}

%% file: intertwining/qcombo.tex
\section{\texorpdfstring{$q$}{q}-Combinatorial Inputs} \label{sec:qcombo}
In this section, we discuss the eigenvalues of some antitriangular matrices. Essentially the only method in the literature to access the eigenvalues of an antitriangular matrix is to do some educated guessing in order to make the given matrix upper-triangular. See \cite{britnell-antitriangular} for a thorough discussion of a special case; the work in this subsection can be seen as a $q$-analogue for some of their results. In \Cref{subsec:q-ids,subsec:helper}, we discuss some purely combinatorial inputs into our main results. Notably, these subsections will not use the notation of \Cref{sec:rep-theory}, and $q$ will be treated as a free variable.

\subsection{A Couple \texorpdfstring{$q$}{q}-Identities} \label{subsec:q-ids}
This subsection records two $q$-identities used later on. Throughout, we freely use the packages \texttt{qZeil} and \texttt{qMultiSum} developed by Axel Riese. See \cite{riese-zeil,riese-multisum} for a description of these packages, and see \cite{qmultisum-code} for the Mathematica notebook used in the proofs.

The following identity is used for the linear groups.
\begin{proposition} \label{prop:gl-q-identity}
    For any nonnegative integers $m$ and $n$, we have
    \begin{align*}
        &q^{-m^{2}+mn}\sum_{i=0}^{m}\left(-1\right)^{i}q^{\frac{1}{2}i\left(i-1\right)-ni}\frac{\left(q;q\right)_m^{2}}{\left(q;q\right)_i\left(q;q\right)_{m-i}^{2}} \\
        ={}&\sum_{i+j+k=n}\left(-1\right)^{i}q^{\frac{1}{2}i\left(i-1\right)-mi}\frac{\left(q;q\right)_n}{\left(q;q\right)_i\left(q;q\right)_j\left(q;q\right)_k}.
    \end{align*}
\end{proposition}
\begin{proof}
    Let the left-hand side be $L_{m,n}(q)$ and the right-hand side be $R_{m,n}(q)$ so that we want to show that $L_{m,n}(q)=R_{m,n}(q)$. We will show that $L_{m,n}(q)$ and $R_{m,n}(q)$ satisfy the same recurrence in $n$ and then check that $L_{m,n}(q)=R_{m,n}(q)$ for some small $n$. With this outline in mind, we have the following steps.
    \begin{enumerate}
        \item After some rearranging, \texttt{qMultiSum} shows that $n\ge0$ makes
        \[R_{m,n+2}\left(q\right)+\left(q^{1+n-m}-2\right)R_{m,n+1}\left(q\right)-\left(q^{1+n}-1\right)R_{m,n}\left(q\right)\]
        vanish. We would like to show that $L_{m,n}(q)$ satisfies the same recurrence in $n$. Define $\widetilde L_{m,n}(q)$ to be
        \[L_{m,n+2}\left(q\right)+\left(q^{1+n-m}-2\right)L_{m,n+1}\left(q\right)-\left(q^{1+n}-1\right)L_{m,n}\left(q\right),\]
        which we would like to vanish. Then \texttt{qZeil} is able to show that $q^{m^2-mn}\widetilde L_{m,n}(q)$ vanishes after some rearranging.
        \item It remains to check that $L_{m,n}(q)=R_{m,n}(q)$ for $n\in\{0,1\}$. For $n=0$, \texttt{qZeil} shows $L_{m,0}(q)=1$, which agrees with $R_{m,0}(q)$. For $n=1$, \texttt{qZeil} shows
        \[L_{m,1}(q)=\frac{2q^m-1}{2q^m-q}L_{m-1,1}(q)\]
        for $m\ge1$ and checks that $R_{m,1}(q)$ satisfies the same recurrence in $m$. So we complete the proof upon computing $L_{0,1}(q)=R_{0,1}(q)=1$.
        \qedhere
    \end{enumerate}
\end{proof}
The following identity is used for the symplectic and orthogonal groups.
\begin{prop} \label{prop:sp-q-identity}
    For any nonnegative integers $m,n\in\ZZ$, we have
    \begin{align*}
        & q^{\frac{-m^2+m}2+mn}\sum_{i=0}^{\floor{m/2}}(-1)^iq^{i(i-1)-2in}\frac{(q;q)_m}{\left(q^2;q^2\right)_i(q;q)_{m-2i}} \\
        ={}& \sum_{j=0}^n(-1)^jq^{j(j-m)}\frac{\left(q^2;q^2\right)_{n}}{\left(q^2;q^2\right)_j(q;q)_{n-j}}.
    \end{align*}
\end{prop}
\begin{proof}
    Let the left-hand side be $L_{m,n}(q)$, and let the right-hand side be $R_{m,n}(q)$. The proof is essentially the same as in \Cref{prop:gl-q-identity}: we will show that $L_{m,n}(q)$ and $R_{m,n}(q)$ satisfy the same recurrence in $n$ and then check that $L_{m,n}(q)=R_{m,n}(q)$ for some small $n$.
    \begin{enumerate}
        \item The package \texttt{qZeil} shows that $n\ge2$ has
        \[R_{m,n}\left(q\right)+\frac{\left(q^{2n}-q^{m+1}-q^{m+2}\right)}{q^{m+1}}R_{m,n-1}\left(q\right)+q\left(1-q^{2n-2}\right)R_{m,n-2}\left(q\right)\]
        vanishes. As before, we define $\widetilde L_{m,n}(q)$ as the same expression above but replacing $R$s with $L$s, and we would like to show that $\widetilde L_{m,n}(q)$ vanishes. Well, \texttt{qZeil} is able to show this after a little rearrangement.
        \item It remains to check that $L_{m,n}(q)=R_{m,n}(q)$ for $n\in\{0,1\}$. For $n=0$, \texttt{qZeil} shows that $L_{m,0}(q)=1$, which agrees with $R_{m,0}(q)$. For $n=1$, \texttt{qZeil} shows that
        \[L_{m,1}(q)=\frac{q-q^m-q^{m+1}}{q^2-q^m-q^{m+1}}L_{m-1,1}(q)\]
        and checks that $R_{m,1}(q)$ satisfies the same recurrence. Thus, it is enough to check that $L_{0,1}(q)=R_{0,1}(q)=1$.
        \qedhere
    \end{enumerate}
\end{proof}

\subsection{Eigenvalues for Linear Groups}
We continue with the notation of \Cref{sec:rep-theory} with $G\in\{\GL_{2n},\SL_{2n}\}$. In this subsection, we will compute the eigenvalues of the intertwining operator when $\beta=1$; if $\beta^2=1$ while $\beta\ne1$, then the eigenvalues are straightforward to compute from \Cref{prop:quadratic-matrix}.

For expositional reasons, we begin by computing the eigenvalues of a certain helper matrix.
\begin{proposition} \label{prop:gl-helper}
    Fix a positive integer $n$. For indices $i,j\in\{0,1,\ldots,n\}$ such that $i+j-n\ge0$, define
    \[\varepsilon_A(i,j)\coloneqq(-1)^{i+j-n},\]
    and
    \[Q_A(i,j)\coloneqq q^{\binom{i+j-n+1}2-\left(i+1\right)^{2}},\]
    and
    \[R_A(i,j)\coloneqq\frac{(q;q)_i^{2}}{(q;q)_{n-j}^{2}(q;q)_{i+j-n}},\]
    and define $R_A(i,j)=0$ for other $i$ and $j$. Then the matrix $A\coloneqq[\varepsilon_A(i,j)Q_A(i,j)R_A(i,j)]_{0\le i,j\le n}$ is diagonalizable with eigenvalues
    \[\left\{(-1)^{n-i}q^{\binom{i+1}2-\binom{n+2}2}:0\le i\le n\right\}.\]
\end{proposition}
\begin{proof}
    This is essentially equivalent to \Cref{prop:gl-q-identity}. For indices $i,j\in\{0,1,\ldots,n\}$ with $j\ge i$, define
    \[\varepsilon_B(i,j)\coloneqq(-1)^i,\]
    and
    \[Q_B(i,j)\coloneqq q^{-(n+1)(j+1)+\binom{i+1}2},\]
    and
    \[R_B(i,j)\coloneqq\sum_{k=0}^{j-i}\frac{(q;q)_j}{(q;q)_i(q;q)_k(q;q)_{j-i-k}},\]
    and define $R_B(i,j)=0$ for other $i$ and $j$. Then we claim that the matrix $B\coloneqq[\varepsilon_B(i,j)Q_B(i,j)R_B(i,j)]_{0\le i,j\le n}$ is similar to $A$, which completes the proof upon reading off the diagonal entries of $B$.

    It remains to show $A\sim B$. We will show $M^{-1}AM=B$, where $M$ has entries
    \[M_{ij}\coloneqq q^{-(i+1)(j+1)}\]
    for $i,j\in\{0,1,\ldots,n\}$. Because $M$ is invertible, it suffices to show $AM=MB$. Thus, for indices $i$ and $k$, we want $(AM)_{ik}=(MB)_{ik}$. Because $A_{ij}$ will vanish unless $i+j\ge n$, and $B_{jk}$ will vanish unless $j\le k$, we see we are asking for
    \[\sum_{j=0}^iA_{i,n-i+j}M_{n-i+j,k}\stackrel?=\sum_{j=0}^kM_{ij}B_{jk}.\]
    Upon plugging in our definitions and simplifying, this reduces to \Cref{prop:gl-q-identity}.
\end{proof}
\begin{theorem} \label{thm:eigens-gl}
    Take $G\in\{\GL_{2n},\SL_{2n}\}$. Fix a character $\chi\colon P\to\CC^\times$, which we write as $\chi=(\alpha\circ m)(\beta\circ\chi_{\det})$. Suppose $\beta=1$ so that $\chi=\chi^J$. Then the operator $I$ on $\Ind_P^G\chi$ is diagonalizable and has eigenvalues given by
    \[\left\{(-1)^{n-i}q^{\binom n2+\binom{i+1}2}:0\le i\le n\right\}.\]
\end{theorem}
\begin{proof}
    Identify $I$ with its matrix representation given in \Cref{prop:trivial-matrix-coeffs}. We apply \Cref{prop:gl-helper} after conjugating $I$ by the matrix $T$ with entries
    \[T_{ij}=\begin{cases}
        -1 & \text{if }i+j=n-1, \\
        1 & \text{if }i+j=n, \\
        0 & \text{otherwise},
    \end{cases}\]
    where $i,j\in\{0,\ldots,n\}$. Now, define $A$ as in \Cref{prop:gl-helper} with $n$ replaced with $n-1$. Then we claim that
    \begin{equation}
        TIT^{-1} \stackrel?= q^{n^2}\begin{bmatrix}
            -\sigma A\sigma \\
            (1,\ldots,1) & 1
        \end{bmatrix}, \label{eq:conj-intertwining-gl}
    \end{equation}
    where $\sigma$ is the permutation sending $e_i\mapsto e_{n-1-i}$ for all $i\in\{0,\ldots,n-1\}$, and $(1,\ldots,1)$ is a row vector consisting of all $1$s. Before proving the claim, we explain how it implies the theorem. Taking the eigenvalues of \Cref{prop:gl-helper} (and some simplification) provides the needed eigenvalues; diagonalizability follows because all eigenvalues are distinct.

    It remains to show \eqref{eq:conj-intertwining-gl}. It's enough to show $TI=q^{n^2}\begin{bsmallmatrix}
        -\sigma A\sigma \\ (1,\ldots,1) & 1
    \end{bsmallmatrix}T$ because $T$ is invertible. Choosing some indices $i$ and $k$,
    we see that we want to show that the $(i,k)$ entries are equal, which amounts to checking
    \begin{align*}
        & I_{n-i,k}-1_{i<n}I_{n-i-1,k} \\
        ={}& q^{n^2}\left(\begin{bmatrix}
            -\sigma A\sigma \\
            (1,\ldots,1) & 1
        \end{bmatrix}_{i,n-k}-1_{k<n}\begin{bmatrix}
            -\sigma A\sigma \\
            (1,\ldots,1) & 1
        \end{bmatrix}_{i,n-k-1}\right)
    \end{align*}
    after expanding out the definition of $T$.
    We verify this by rather tedious casework on $i$ and $k$. Denote the left-hand side by $L$ and the right-hand side by $R$.
    \begin{itemize}
        \item Suppose $i=k=n$. Then the definition of $I$ yields $L=q^{n^2}$, and we find $R=q^{n^2}$ as well.
        \item Suppose $i=n$ but $k<n$. Then we see $L=0$, and one can check that $R=q^{n^2}(1-1)=0$.
        \item Suppose $i<n$ but $k=n$. Then $L$ and $R$ equal
        \[(-1)^{n-i}q^{n^2-\binom{n-i+1}2}(q;q)_{n-i-1}.\]
        \item Suppose $k<i<n$. Then $(n-i)+k<n$ and $(n-1-i)+(n-1-(n-k-1))<n-1$, so all coefficients vanish.
        \item Suppose $i=k<n$. Then $L=q^{n^2-(n-i)^2}=R$.
        \item Suppose $i<k<n$. Then $L$ and $R$ equal
        \[(-1)^{k-i}q^{n^2-(n-i)^2+\binom{k-i}2}\frac{(q;q)_{n-i-1}^2}{(q;q)_{n-k}^2(q;q)_{k-i}}\left(1-2q^{n-i}+q^{2n-i-k}\right).\]
    \end{itemize}
    The above casework completes the proof.
\end{proof}

\subsection{A Helper Matrix} \label{subsec:helper}
For this subsection, $q$ will return to being a free variable. Akin to \Cref{prop:gl-helper}, we describe a general helper matrix which shows up in the upper-triangularization for the groups $G\in\{\GO_{2n},\O_{2n},\GSp_{2n},\Sp_{2n}\}$, so we handle it here. For some fixed nonnegative integer $n$ and sign $\varepsilon\in\{\pm1\}$ and $a\in\CC$, we select indices $0\le i,j\le n$ such that $i+j-n$ is an even nonnegative integer and define
\[\varepsilon_A(i,j)\coloneqq\varepsilon^{\frac{i-j-n}2}(-1)^{\frac{i+j-n}2}\]
and
\[Q_A(i,j)\coloneqq q^{-\binom{i+a}2+\frac{i+j-n}{2}\left(\frac{i+j-n}{2}+a-1\right)}\]
and
\[R_A(i,j)\coloneqq\frac{(q;q)_i}{(q^2;q^2)_{(i+j-n)/2}(q;q)_{n-j}},\]
and define $R_A(i,j)=0$ for other indices $i$ and $j$. Then  set $A\coloneqq[\varepsilon_A(i,j)Q_A(i,j)R_A(i,j)]_{0\le i,j\le n}$. This $(n+1)\times(n+1)$ matrix will be used in approximately the same way we used \Cref{prop:gl-helper}. In particular, we want to understand its eigenvalues.

We now compute the eigenvalues of $A$ when $n$ is even.
\begin{proposition} \label{prop:helper-matrix-even}
    Define $n$, $\varepsilon$, $a$, and $A$ as above. If $n=2m$, then the antitriangular matrix $[A(i,j)]_{0\le i,j\le n}$ is diagonalizable with eigenvalue multiset
    \[\left\{\varepsilon^{m}(-1)^{\floor{\frac i2}}q^{-\binom{a+m}2-\binom{m+1}2+\left(m-\floor{\frac i2}\right)^2}:0\le i\le n\right\}.\]
\end{proposition}
\begin{proof}
    We follow \Cref{prop:gl-helper}. For indices $0\le i,j\le n$ such that $j-i$ is a nonnegative even integer, define
    \[\varepsilon_B(i,j)\coloneqq\varepsilon^{m}\left(-1\right)^{\floor{\frac{i}{2}}}\]
    and
    \[Q_B(i,j)\coloneqq q^{-\binom{a+m}2-\binom{m+1}2+\left(m-\floor{\frac{i}{2}}\right)^{2}-\frac{j-i}{2}\left(2m-i+2\floor{\frac{i}{2}}\right)}\]
    and
    \[R_B(i,j)\coloneqq\frac{(q^2;q^2)_{\floor{j/2}}}{(q^2;q^2)_{\floor{i/2}}(q;q)_{(j-i)/2}},\]
    and define $R_B(i,j)=0$ for other indices $i$ and $j$. Then we claim that $A$ is similar to the matrix $B\coloneqq[\varepsilon_B(i,j)Q_B(i,j)R_B(i,j)]_{0\le i,j\le n}$, which will complete the proof upon reading off the diagonal entries of $B$. Notably, even though some eigenvalues are equal, $B$ splits into a direct sum of operators on the even basis vectors and on the odd basis vectors, and the operators on these subspaces have distinct eigenvalues.

    It remains to show the claim. We will show $M^{-1}AM=B$, where $M$ is defined by
    \[M_{ij}\coloneqq\varepsilon^{\floor{i/2}}q^{\floor{i/2}(2\floor{j/2}+a)}\]
    for indices $i$ and $j$ such that $i\equiv j\pmod2$ and zero elsewhere. Because $M$ is invertible, it is enough to show $AM=MB$. Thus, for indices $i$ and $k$, we want to show that $(AM)_{ik}=(MB)_{ik}$. 
    If $i$ and $k$ fail to have the same parity, then we note $(AM)_{ik}=(MB)_{ik}=0$ because $A$, $M$, and $B$ all send even (and odd) basis vectors to linear combinations of even (and odd) basis vectors. Thus, we may assume that $i\equiv k\pmod2$. Now, we are left to verify the identity
    \[\sum_{j=0}^nA_{ij}M_{jk}\stackrel?=\sum_{j=0}^nM_{ij}B_{jk}.\]
    Note $A_{ij}$ will vanish unless $i+j-n$ is a nonnegative even integer, and $B_{jk}$ will vanish unless $k-j$ is a nonnegative even integer. Thus, by reindexing, it suffices to check
    \[\sum_{j=0}^{\floor{i/2}}A_{i,n-i+2j}M_{n-i+2j,k}\stackrel?=\sum_{j=0}^{\floor{k/2}}M_{i,k-2\floor{k/2}+2j}B_{k-2\floor{k/2}+2j,k}.\]
    Upon plugging in our definitions and simplifying, this reduces to \Cref{prop:sp-q-identity}.
\end{proof}
One can upgrade the eigenvalue computations in the even case to the odd case as follows.
\begin{proposition} \label{prop:helper-matrix-odd}
    Define $n$, $\varepsilon$, $a$, and $A$ as above. If $n$ is odd with $n=2m+1$, then the antitriangular matrix $[A(i,j)]_{0\le i,j\le n}$ is diagonalizable with eigenvalues
    \[\left\{\pm\sqrt\varepsilon q^{-\frac a2-\binom{a+m}2-\binom{m+1}2+(m-i)^2-i}:0\le i\le m\right\}.\]
\end{proposition}
\begin{proof}
    Fixing $\varepsilon$ and $a$ but letting $n$ vary, denote the defined $(n+1)\times(n+1)$ matrix by $A_n$. We are interested in the eigenvalues of $A_{2m+1}$ for some $m\ge0$. For some $m\ge0$, note that $A_{2m}$ sends even (and odd) basis vectors to a linear combination of even (and odd) basis vectors, so we let $A_{2m}^+$ (and $A_{2m}^-$) denote the submatrices of $A_{2m}$ consisting of the even (and odd) and columns and rows (respectively). Thus, by rearranging the rows and columns of $A_{2m}$, we see that $A_{2m}$ is similar to 
    \[\begin{bmatrix}
        A_{2m}^+ \\ & A_{2m}^-
    \end{bmatrix}.\]
    On the other hand, we see that $A_{2m+1}$ sends even (and odd) basis vectors to a linear combination of odd (and even) basis vectors. Defining $A_{2m+1}^+$ (and $A_{2m+1}^-$) to be the submatrix consisting of the even (odd) columns and odd (even) rows, we thus see that $A_{2m+1}$ is similar to
    \[\begin{bmatrix}
        & A_{2m+1}^- \\ A_{2m+1}^+
    \end{bmatrix}.\]
    Now, for each $n$, we note that $A_n$ is a submatrix of $\varepsilon A_{n+1}$, so keeping track of parities reveals that $A_n^+=\varepsilon A_{n+1}^-$. Thus, $A_{2m+1}$ is similar to the antitriangular matrix
    \[\varepsilon\begin{bmatrix}
        & A_{2m}^+ \\ A_{2m+2}^-
    \end{bmatrix}.\]
    The proof of \Cref{prop:helper-matrix-even} explains that $A_{2m}^+$ is similar to an upper-triangular matrix with diagonal entries
    \[\left\{\varepsilon^{m+1}(-1)^{i}q^{-\binom{a+m}2-\binom{m+1}2+\left(m-i\right)^2}:0\le i\le m\right\},\]
    and $A_{2m+2}^-$ is similar to an upper-triangular matrix with diagonal entries
    \[\left\{\varepsilon^{m}(-1)^{i}q^{-\binom{a+m+1}2-\binom{m+2}2+(m+1-i)^2}:0\le i\le m\right\}\]
    where we conjugate by the same $(m+1)\times(m+1)$ matrix! Thus, viewing $A_{2m+1}$ as an $(m+1)\times(m+1)$ matrix with entries that are $2\times2$ (block) matrices, we see that $A_{2m+1}$ is similar to a (block) upper-triangular matrix with diagonal entries
    \[\left\{\varepsilon^{m+1}(-1)^iq^{-\binom{a+m}2-\binom{m+1}2+(m-i)^2}\begin{bmatrix}
        & 1 \\
        \varepsilon q^{-a-2i}
    \end{bmatrix}:0\le i\le m\right\}.\]
    The result follows from diagonalizing these $2\times2$ matrices and noting that all the eigenvalues are distinct.
\end{proof}

\subsection{Eigenvalues for Orthogonal Groups}
We continue with the notation of \Cref{sec:rep-theory}, taking $G\in\{\GO_{2n},\O_{2n}\}$. We (essentially) begin with the case $\beta=1$.
\begin{theorem} \label{thm:o-trivial-eigens}
    Take $G\in\{\GO_{2n},\O_{2n}\}$. Fix a character $\chi\colon P\to\CC^\times$, which we write as $\chi=(\alpha\circ m)(\beta\circ\chi_{\det})$. Suppose either that $\beta=1$, or $\beta^2=1$ for $G=\O_{2n}$.
    \begin{listalph}
        \item If $n=2m$ is even, then the intertwining operator $I$ on $\Ind_P^G\chi$ is diagonalizable and has eigenvalue multiset given by
        \[\left\{(-1)^{m-\floor{\frac i2}}q^{m(m-1)+\floor{\frac i2}^2}:1\le i\le n+1\right\}.\]
        \item If $n=2m+1$ is odd, then the intertwining operator $I$ on $\Ind_P^G\chi$ is diagonalizable and has eigenvalues given by
        \[\left\{\pm q^{m^2+i(i+1)}:0\le i\le m\right\}.\]
    \end{listalph}
\end{theorem}
\begin{proof}
    The assumptions imply that the intertwining operator $I$ has a uniform matrix representation given in \Cref{prop:trivial-matrix-coeffs,prop:quadratic-matrix}. Now, we define $A$ as in \Cref{subsec:helper} with $(n,\varepsilon,a)=(n,1,0)$ so that $q^{\binom n2}A$ is the matrix representation of $I$. The result now follows from combining \Cref{prop:helper-matrix-even,prop:helper-matrix-odd} and simplifying the eigenvalues.
\end{proof}
It remains to cover the case where $\beta^2=1$ but $\beta\ne1$ when $G=\GO_{2n}$. This will follow by submatrix considerations via \Cref{lem:general-from-special-matrix}.
\begin{theorem} \label{thm:go-quad-eigens}
    Take $G=\GO_{2n}$. Fix a character $\chi\colon P\to\CC^\times$, which we write as $\chi=(\alpha\circ m)(\beta\circ\chi_{\det})$. Assume that $\beta^2=1$ but $\beta\ne1$.
    \begin{listalph}
        \item If $n=2m$ is even, then the intertwining operator $I\circ I$ on $\Ind_P^G\chi$ is diagonalizable and has eigenvalues
        \[\left\{q^{2m(m-1)+2i^2}:0\le i\le m\right\}.\]
        \item If $n=2m+1$ is odd, then the intertwining operator $I\circ I$ on $\Ind_P^G\chi$ is diagonalizable and has eigenvalues
        \[\left\{q^{2m^2+2i(i+1)}:0\le i\le m\right\}.\]
    \end{listalph}
\end{theorem}
\begin{proof}
    We combine the computations of \Cref{thm:o-trivial-eigens} with \Cref{lem:general-from-special-matrix}. Let $I^0$ be the $(n+1)\times(n+1)$ matrix representation of the corresponding operator for $\O_{2n}$, and let $I^+$ and $I^-$ be the submatrices of $I^0$ given in \Cref{lem:general-from-special-matrix} which are the matrix representations of $I$ on $\left(\Ind_P^G\chi\right)^\chi\to\left(\Ind_P^G\chi^J\right)^\chi$ and $\left(\Ind_P^G\chi^J\right)^\chi\to\left(\Ind_P^G\chi\right)^\chi$, respectively. We must compute the eigenvalues of $I^-\circ I^+$. We now handle the even and odd cases separately.
    \begin{listalph}
        \item If $n=2m$ is even, then $I^+$ and $I^-$ are both the submatrix of $I^0$ consisting of the even rows and columns. Tracking through the proof of \Cref{thm:o-trivial-eigens} (and notably its input \Cref{prop:helper-matrix-even}), we will show that $I^+$ and $I^-$ are both diagonalizable with eigenvalues
        \[\left\{(-1)^{i}q^{m(m-1)+(m-i)^2}:0\le i\le m\right\},\]
        which completes the proof upon squaring our eigenvalues. Indeed, defining $A$ as in \Cref{thm:o-trivial-eigens}, we see that $I^+=I^-$ is a submatrix of $q^{\binom n2}A$, and the upper-triangularization of $A$ given in \Cref{prop:helper-matrix-even} restricts to an upper-triangularization of $I^+$ and $I^-$. In particular, we can read off the eigenvalues by taking the correct entries from (a) of \Cref{thm:o-trivial-eigens} (or equivalently, \Cref{prop:helper-matrix-even}).
        \item If $n=2m+1$ is odd, then $I^+$ is the submatrix of $I^0$ consisting of the even columns and odd rows, and $I^-$ is the submatrix of $I^0$ consisting of the odd columns and even rows. Arguing as above, we define $A$ as in \Cref{thm:o-trivial-eigens} so that we see that $I^0$ is similar to $q^{\binom{n}2}A$, which in turn \Cref{prop:helper-matrix-odd} explains is similar to the block diagonal matrix with block diagonal given by the $2\times2$ matrices
        \[\left\{(-1)^iq^{\binom{2m+1}2-\binom{m}2-\binom{m+1}2+(m-i)^2}\begin{bmatrix}
            & 1 \\
            q^{-2i}
        \end{bmatrix}:0\le i\le m\right\}.\]
        All stated similarities preserve the even and odd subspaces, so we see that the matrices $I^+$ and $I^-$ will be similar to the diagonal matrices achieved by reading off the block diagonal in the matrix described above. As such, computing the composite $I^-\circ I^+$ tells us that the eigenvalues are
        \[\left\{q^{2m^2+2i(i+1)}:0\le i\le m\right\}\]
        after a little simplification.
        \qedhere
    \end{listalph}
\end{proof}

\subsection{Eigenvalues for Symplectic Groups}
We continue with the notation of \Cref{sec:rep-theory}, taking $G\in\{\GSp_{2n},\Sp_{2n}\}$. We begin with the case $\beta=1$.
\begin{theorem}
    Take $G\in\{\GSp_{2n},\Sp_{2n}\}$. Fix a character $\chi\colon P\to\CC^\times$, which we write as $\chi=(\alpha\circ m)(\beta\circ\chi_{\det})$. Assume that $\beta=1$ so that $\chi=\chi^J$.
    \begin{listalph}
        \item If $n=2m$ is even, then the intertwining operator $I$ on $\Ind_P^G\chi$ is diagonalizable and has eigenvalues given by
        \[\left\{\pm q^{m^2+i(i+1)}:0\le i\le m-1\right\}\sqcup\left\{q^{\binom{2m+1}2}\right\}.\]
        \item If $n=2m+1$ is odd, then the intertwining operator $I$ on $\Ind_P^G\chi$ is diagonalizable and has eigenvalue multiset given by
        \[\left\{-(-1)^{m-\floor{\frac i2}}q^{m(m+1)+\floor{\frac i2}^2}:1\le i\le n+1\right\}.\]
    \end{listalph}
\end{theorem}
\begin{proof}
    This proof follows \Cref{thm:eigens-gl} upon replacing the computations of \Cref{prop:gl-helper} with those of \Cref{prop:helper-matrix-even,prop:helper-matrix-odd}.
    Identify $I$ with its matrix representation. We will apply \Cref{prop:helper-matrix-even,prop:helper-matrix-odd} after conjugating $I$ by the $(n+1)\times(n+1)$ matrix $T$ defined by
    \[T_{ij}=\begin{cases}
        -1 & \text{if }i+j=n-1, \\
        1 & \text{if }i+j=n, \\
        0 & \text{otherwise},
    \end{cases}\]
    where $i,j\in\{0,\ldots,n\}$.
    Now, define $A$ as in \Cref{subsec:helper} with $(n,\varepsilon,a)=(n-1,1,2)$. Then we claim that
    \begin{equation}
        TIT^{-1}\stackrel?=q^{\binom{n+1}2}\begin{bmatrix}
            -\sigma A\sigma \\
            (1,\ldots,1) & 1
        \end{bmatrix}, \label{eq:conjugate-sp-trivial}
    \end{equation}
    where $\sigma$ is the permutation matrix sending $e_i\mapsto e_{n-1-i}$ for all $i\in\{0,\ldots,n-1\}$, and $(1,\ldots,1)$ is a row vector consisting of all $1$s. Before proving the claim, we explain how it implies the result. Using the eigenvalues of \Cref{prop:helper-matrix-even,prop:helper-matrix-odd} (and some simplification) checks that the eigenvalues in the theorem are correct. Because $A$ is diagonalizable and does not have $-1$ as an eigenvalue, diagonalizability follows.

    It remains to show \eqref{eq:conjugate-sp-trivial}. It is enough to check $TI=q^{\binom{n+1}2}\begin{bsmallmatrix}
        -\sigma A\sigma \\ (1,\ldots,1) & 1
    \end{bsmallmatrix}T$ because $T$ is invertible, which one can do using the same sort of casework engaged in \Cref{thm:eigens-gl}. We will not write out the results for brevity.
\end{proof}
We now move on to the case where $\beta^2=1$ but $\beta\ne1$. We first treat $\Sp_{2n}$.
\begin{theorem} \label{thm:sp-quadratic-eigens}
    Take $G=\Sp_{2n}$. Fix a character $\chi\colon P\to\CC^\times$, which we write as $\chi=\beta\circ\chi_{\det}$. Suppose $\beta^2=1$ but $\beta\ne1$ so that $\chi=\chi^J$.
    \begin{listalph}
        \item If $n=2m$ is even, then the intertwining operator $I$ on $\Ind_P^G\chi$ is diagonalizable and has eigenvalues
        \[\left\{\beta(-1)^m(-1)^{m-\floor{\frac i2}}q^{m^2+\floor{\frac{i}2}^2}:0\le i\le n\right\}.\]
        \item If $n=2m+1$ is odd, then the intertwining operator $I$ on $\Ind_P^G\chi$ is diagonalizable and has eigenvalues
        \[\left\{\pm\sqrt{\beta(-1)}q^{m(m+1)+i(i+1)+\frac12}:0\le i\le m\right\}.\]
    \end{listalph}
\end{theorem}
\begin{proof}
    Define $A$ as in \Cref{subsec:helper}, with $(n,\varepsilon,a)=(n,\beta(-1),1)$. Then the matrix representation of $I$ is the matrix $q^{\binom{n+1}2}A$. The result now follows by plugging into the eigenvalue computations of \Cref{prop:helper-matrix-even,prop:helper-matrix-odd} and simplifying.
\end{proof}
As in \Cref{thm:go-quad-eigens}, we now use submatrix arguments to compute the eigenvalues for $\GSp_{2n}$.
\begin{theorem}
    Take $G=\GSp_{2n}$. Fix a character $\chi\colon P\to\CC^\times$, which we write as $\chi=(\alpha\circ m)(\beta\circ\chi_{\det})$. Suppose that $\beta^2=1$ but $\beta\ne1$.
    \begin{listalph}
        \item If $n=2m$ is even, then the intertwining operator $I\circ I$ on $\Ind_P^G\chi$ is diagonalizable and has eigenvalues
        \[\left\{q^{2m^2+2i^2}:0\le i\le m\right\}.\]
        \item If $n=2m+1$ is odd, then the intertwining operator $I\circ I$ on $\Ind_P^G\chi$ is diagonalizable and has eigenvalues
        \[\left\{\beta(-1)q^{2m(m+1)+2i(i+1)+1}:0\le i\le m\right\}.\]
    \end{listalph}
\end{theorem}
\begin{proof}
    The argument is exactly the same as \Cref{thm:go-quad-eigens} upon replacing the computations of \Cref{thm:o-trivial-eigens} with \Cref{thm:sp-quadratic-eigens}.
\end{proof}